\documentclass[11pt,twoside]{article}

\usepackage{fullpage}

\input{command-template}

\begin{document}

\begin{center}

  {\bf{\LARGE{Screening Cut Generation for Sparse Ridge Regression}}}

\vspace*{.2in}

{\large{
\begin{tabular}{ccc}
Haozhe Tan$^{1}$, Guanyi Wang$^{1}$
\end{tabular}
}}
\vspace*{.2in}

\begin{tabular}{c}
$^{1}$Department of Industrial Systems Engineering and Management, \\
National University of Singapore\\

\end{tabular}

\vspace*{.2in}

\today

\vspace*{.2in}

\begin{abstract}
Sparse ridge regression is widely utilized in modern data analysis and machine learning. However, computing globally optimal solutions for sparse ridge regression is challenging, particularly when samples are arbitrarily given or generated under weak modeling assumptions. This paper proposes a novel cut-generation method, Screening Cut Generation (SCG), to eliminate non-optimal solutions for arbitrarily given samples. In contrast to recent safe variable screening approaches, SCG offers superior screening capability by identifying whether a specific $\{\pm 1\}$ combination of multiple features (binaries) lies in the set of optimal solutions. This identification is based on a convex relaxation solution rather than directly solving the original sparse ridge regression. Hence, the cuts generated by SCG can be applied in the pre-processing step of branch-and-bound and its variants to construct safe outer approximations of the optimal solution set. Numerical experiments are reported to validate the theoretical results and demonstrate the efficiency of SCG, particularly in hard real instances and synthetic instances with high dimensions, low ridge regularization parameters, or challenging modeling assumptions. 
\end{abstract}
\end{center}

\section{Introduction} \label{sec:intro}

Sparse ridge regression is a preponderantly used tool for feature selection while ensuring interpretability and generalization in various domains, including biological genetic analysis \citep{ogutu2012genomic}, sparsity-aware learning in compressed sensing \citep{theodoridis2014sparsity}, healthcare diagnostics \citep{gramfort2013time}, and consumer defection prediction for cloud-based software \citep{kuswanto2015logistic}, to name but a few. The problem is formulated as follows:  
\begin{equation}
    \label{eq:original-formulation}
    v^* := \min_{\bm{\beta} \in \mathbb{R}^d} ~ \mathcal{L}(\bm{\beta}) + \gamma \vecnorm{\bm{\beta}}{2}^2 \quad \text{s.t.} \quad \vecnorm{\bm{\beta}}{0} \leq k ~~,  \tag{$\mathcal{P}$} 
\end{equation}
where $\mathcal{L}: \mathbb{R}^d \to \mathbb{R}$ represents a given convex loss, $\gamma > 0$ is a parameter for $\ell_2$-norm ridge regularizer that promotes coefficient shrinkage, and the sparsity constraint $\vecnorm{\bm{\beta}}{0} \leq k$ controls the size of support set $\supp(\bm{\beta}) := \{i \in [\![d]\!] ~|~ \beta_i \neq 0 \}$ by at most $k$ with $\vecnorm{\bm{\beta}}{0} := |\supp(\bm{\beta})| \leq k$. In Mixed-Integer Programming (MIP) community, the sparsity constraint $\vecnorm{\bm{\beta}}{0} \leq k$ can be reformulated by introducing binary variables $\{z_i\}_{i = 1}^d \in \{0,1\}^d$ as $Z^k := \{ (\bm{\beta}, \bm{z}) \in \mathbb{R}^d \times \{0,1\}^d ~|~ \bm{1}^{\top} \bm{z} \leq k, ~ \beta_i (1 - z_i) = 0 ~ \forall ~ i \in [\![d]\!] \}.$

Given the NP-hardness of general sparse ridge regression \citep{chen2019approximation} and its variants, a substantial body of research has focused on developing either inexact (approximate) iterative algorithms (i.e., local search, first-order based algorithm) or exact algorithms based on Branch-and-Bound (BnB) frameworks tailored to the original formulation \eqref{eq:original-formulation}, as detailed in the literature review (see Section~\ref{sec:preliminary-review}). A significant advancement in exact algorithm studies involves using (safe) variable screening approaches in the pre-processing step of BnB and its variants. Specifically, (safe) variable screening aims to determine whether a binary variable is guaranteed to be zero or one in an optimal solution of the original problem \eqref{eq:original-formulation}, based on optimal solutions from convex relaxations rather than solving \eqref{eq:original-formulation} directly. Notably, for general MIPs (also \eqref{eq:original-formulation}), identifying a single binary variable roughly halves the number of feasible binary solutions. Such a reduction yields significant computational benefits, especially for enumeration-based optimization methods (e.g., BnB), since BnB usually scales exponentially with the number of binaries identified or eliminated. \emph{Therefore, the effectiveness of variable screening approaches critically affects the computational efficiency on solving \eqref{eq:original-formulation}.}

While promising results of safe screening techniques \citep{Ata20,Dez22} shed light on the efficiency and potential applicability of existing screening rules, challenges remain in many instances -- particularly synthetic instances with relatively low ridge regularization parameter regime or signal-to-noise ratios, and real instances in machine learning community (see Section~\ref{sec:numerical} for details). In these instances, a significant part of binary variables remain hard to identify using current screening rules in pre-processing step. These observations therefore lead to the following research question:
\begin{center}
    \emph{How and to what extent can convex relaxation solutions be leveraged to enhance screening approaches?}
\end{center}
As a first step to address the above question, this paper introduces a novel cut-generation method, Screening Cut Generation (SCG). By leveraging fractional convex relaxation solutions, SCG safely eliminates non-optimal but feasible solutions for sparse ridge regression. Unlike existing approaches, SCG determines not only single binary variables but also specific $\{\pm 1\}$ combinations of multiple binaries is in the optimal solution set.  

\subsection{Main Contributions \& Paper Organization} 
Our main contributions are given as follows: Section~\ref{sec:preliminary-review} provides a literature review followed by a brief preliminary on existing optimization algorithms and pre-processing techniques for sparse ridge regression. 
Section~\ref{sec:main-results} presents our main contributions. Specifically, Section~\ref{eq:screening-cut-generation} introduces criteria of the proposed novel cut-generation method, Screening Cut Generation (SCG), which enhances existing variable screening approaches in the pre-processing step of BnB. Notably, we demonstrate that convex relaxation solutions can be more effectively utilized to generate screening cuts that eliminate not only single non-optimal binary variables but also specific $\{\pm 1\}$ combinations of multiple binaries that are not part of any optimal solutions, which generalizes existing screening approaches. Section~\ref{sec:SCG-property} explores the theoretical properties of SCG by establishing necessary and sufficient conditions for certifying minimal (dominating) screening cuts. Additionally, Section~\ref{sec:k-knapsack-polytope} examines connections between SCG and the $k$-cardinality constrained binary knapsack polytope, providing an initial step toward investigating the closure of the proposed SCG cuts, which tackles the proposed research question on to what extent SCG cuts could outer approximate the optimal solution set. Section~\ref{sec:implementation-technique} introduces novel selection criteria for SCG cuts by quantifying their potential screening ability to further improve numerical efficiency, which is validated in Section~\ref{sec:numerical}. Numerical experiments are reported in Section~\ref{sec:numerical} to demonstrate the advantages of the proposed SCG compared to existing baselines. 

\subsection{Notation} 
We use lowercase letters, e.g., $x$, boldface lowercase letters, e.g., $\bm{x}$, and boldface uppercase letters, e.g., $\bm{X}$, to represent scalars, vectors, and matrices, respectively. We use $\bm{1}$ to denote a vector of all ones. Given two integers $a, b$ with $a \leq b$, we use $[\![ a,b ]\!]$ to denote the set of consecutive integers from $a$ to $b$, i.e., $[\![ a,b ]\!] := \{a, \ldots, b\}$, and use $[\![b]\!] := \{1, \ldots, b\}$ as a shorthand notation for $[\![1,b]\!]$. Given a vector $\bm{x} \in \mathbb{R}^d$, for any $i \in [\![d]\!]$, we use $(\bm{x})_i$ to denote the \( i \)-th component with $x_i$ a shorthand notation without further explanations, \( (x)_{[k]} \)  or \( x_{[k]} \) to denote the \( k \)-th largest component, and \( \underline{x} := x_{[d]} \) to be the smallest component (i.e., \( \min_{i \in [\![d]\!]} x_i \)) of \( \bm{x} \). Given two vectors $\bm{x}, \bm{y} \in \mathbb{R}^d$, we use $\bm{x} \circ \bm{y} \in \mathbb{R}^d$ to denote the Hadamard product with $(\bm{x} \circ \bm{y})_i = x_i y_i$ for all $i \in [\![d]\!]$. We use \(|\cdot|\) to denote the cardinality of a set and \(\cdot \setminus \cdot\) to denote the set difference. Given an index subset $S \subseteq [\![d]\!]$, we use \( \bm{x}_S \in \mathbb{R}^{|S|} \) as a subvector of \( \bm{x} \) with its entries restricted on index subset \( S \).  

\section{Preliminary \& Literature Review} \label{sec:preliminary-review}

Sparse ridge regression and its optimization methods have been extensively explored over the past decades, and it is beyond the scope of this work to comprehensively review this vast literature. Instead, this section highlights the studies most relevant to our research. Ridge regression was originally introduced by \cite{hoerl1970ridge} to improve the robustness of regression coefficients $\bm{\beta}$. Subsequently, $\ell_1$ regularization (LASSO) and combined $\ell_1$-$\ell_2$ regularization (elastic net) were developed to promote sparsity and robustness in regression coefficients. In recent years, there has been growing interest in solving sparse ridge regression directly, which can be, in general, separated into the following categories.

The first category focuses on designing approximation algorithms for sparse ridge regression under additional statistical assumptions. Examples include greedy and rounding algorithms \citep{Xie20}, dual iterative hard thresholding \citep{Yuan20}, and projected gradient descent \citep{Con19}. \emph{In comparison, this paper aims to solve \eqref{eq:original-formulation} exactly to optimality without relying on any sample generation or modeling assumptions.}

The second category explores MIP-based approaches for sparse ridge regression. Recent advances \citep{Ber16,Park20,Ber20,Gomez21} have demonstrated promising results in solving sparse ridge regression problems with hundreds of variables. Additionally, the state-of-the-art method \citep{Liu24} introduces a customized branch-and-bound (BnB) framework that integrates fast lower-bound estimation and heuristic search, enabling the resolution of instances with thousands of variables. \emph{In contrast, this paper aims to enhance the pre-processing step of BnB by screening binary variables based on convex relaxation solutions.}

The third category studies variable screening approaches used in the pre-processing step, which reduce the searching space of feasible binary solutions based on convex relaxation solutions. For LASSO-type problems, a series of studies \citep{Lau11,Wang13,Liu14,Fer15,Dan21} have introduced \emph{safe} screening rules, while \cite{Tib10} proposes \emph{strong} rules that are computationally efficient but may occasionally exclude optimal solutions. More recently, \cite{Ata20,Dez22} have developed \emph{exact safe} screening rules for sparse ridge regression. We highlight that the convex relaxation solutions used in our work, as well as in \cite{Ata20,Dez22}, are derived by solving the following perspective transformation of the $\ell_2$-norm regularization term, 

\begin{align}
v_{\tt conic} ~ & := \min_{\bm{t}, \bm{\beta}, \bm{z} \in [0,1]^d} \; \mathcal{L}(\bm{\beta}) + \gamma \bm{1}^{\top} \bm{t} \quad \text{s.t.} \quad \bm{1}^{\top} \bm{z} \leq k,~~t_i \geq \frac{\beta_i^2}{z_i}, ~\forall ~ i \in [\![d]\!] \label{eq:conic-relaxation}\\
&= \max_{\bm{p}} \min_{\bm{\beta}, \bm{z} \in[0,1]^d} \; \mathcal{L}(\bm{\beta}) + \gamma \sum_{i=1}^d \left(p_i \beta_i - \frac{p_i^2 z_i}{4}\right) \quad \text{s.t.} \quad \bm{1}^{\top} \bm{z} \leq k \label{eq:fenchel-relaxation}
\end{align}
where \eqref{eq:conic-relaxation} is shown as a tight convex relaxation of \eqref{eq:original-formulation} \citep{Xie20} and has been widely adopted in the literature \citep{Xie20,Ata20,Liu24}. The reformulation \eqref{eq:fenchel-relaxation} is obtained by taking the Fenchel conjugate on $\beta_i^2 / z_i$, then exchanging the minimization and maximization due to strong duality. Using convex relaxation solutions from solving \eqref{eq:fenchel-relaxation}, \cite{Ata20} show that:

\begin{proposition}[Restatement, Safe Screening Rules from \cite{Ata20}\footnote{We realized that a recent work ``\emph{Logic Rules and Chordal Graphs for Sparse Learning}'' by A. Deza, A. G\'{o}mez, A. Atamt\"{u}rk proposed a refined safe screening rule using techniques about chordal graphs. However, to the best of our knowledge, there is only one presentation slide for the International Symposium on Mathematical Programming (ISMP 2024) without any accessible preprint online. Additionally, based on their slides, we believe our main result (see Theorem~\ref{thm:generalized-cuts-generation-rule} below) is distinct from their main contributions by taking multiple binaries into consideration, and proposes a SCG framework with both theoretical and practical advantages by selecting ``strong'' screening cuts in the pre-processing step. 
}]
    \label{prop:AA-screen-rule}
    Denote the optimal solution of (\ref{eq:fenchel-relaxation}) as $\hat{\bm{\beta}}$ and $\hat{\bm{p}}$. Assume there is no tie in element-wise square of $\hat{\bm{p}}$, i.e. $\hat{\bm{w}}:=\hat{\bm{p}} \circ \hat{\bm{p}}$. Let $v_{\tt ub}$ be any upper bound of $v^*$, usually obtained by any feasible solutions of \eqref{eq:original-formulation}. Then, for any $j \in [\![d]\!]$, the following safe screening rule 
    \begin{equation*}
        z_j = \left\{
        \begin{array}{lll}
            0, & \text{if} \quad \gamma \frac{\hat{w}_j}{4} \leq \gamma \frac{\hat{w}_{[k+1]}}{4} \quad 
             \text{and} \quad \gamma \frac{\hat{w}_{[k]}}{4} - \gamma \frac{\hat{w}_{j}}{4} > \gap \\
            1, & \text{if} \quad \gamma \frac{\hat{w}_{j}}{4} \geq \gamma \frac{\hat{w}_{[k]}}{4} 
             \quad \text{and} \quad -\gamma \frac{\hat{w}_{[k+1]}}{4} + \gamma \frac{\hat{w}_{j}}{4} > \gap
        \end{array}
    \right. ~~ 
\end{equation*} 
identifies whether a binary $z_i$ could be zero or one in optimal solution sets with $\gap := v_{\tt ub} - v_{\tt conic}$. 
\end{proposition}
To be concise, in later discussions, we use \emph{comparison criteria} to denote the first if condition, i.e., $\gamma \frac{\hat{w}_j}{4} \leq \gamma \frac{\hat{w}_{[k+1]}}{4}$ or $\gamma \frac{\hat{w}_{j}}{4} \geq \gamma \frac{\hat{w}_{[k]}}{4}$, respectively; and \emph{reduced-cost criteria} to denote the second if condition, i.e., $\gamma \frac{\hat{w}_{[k]}}{4} - \gamma \frac{\hat{w}_{j}}{4} > \gap$ or $-\gamma \frac{\hat{w}_{[k+1]}}{4} + \gamma \frac{\hat{w}_{j}}{4} > \gap$, respectively. We call the second if condition as \emph{reduced-cost criteria} since its left-hand-side term plays a similar role as the reduced-cost (see discussions after Theorem~\ref{thm:generalized-cuts-generation-rule} for details). As we can observe, the above safe-screening rules have a strong dependence on the $\ell_2$-norm regularization parameter $\gamma$ and the gap between upper and lower bounds of \eqref{eq:original-formulation}. Notably, for sparse ridge regression, its upper bound $v_{\tt ub}$ is obtained by any feasible solution of \eqref{eq:original-formulation} and its lower bound $v_{\tt conic}$ is computed from solving the conic relaxation \eqref{eq:conic-relaxation}, the corresponding $\gap$ is therefore lower bounded by the integrality gap between original \eqref{eq:original-formulation} and corresponding relaxation \eqref{eq:conic-relaxation}.

\section{Main Results} \label{sec:main-results}

\subsection{Screening Cut Generation Rule} \label{eq:screening-cut-generation}

Recall screening rules in Proposition~\ref{prop:AA-screen-rule}, the $j$-th binary variable $z_j$ is identified if both comparison and reduced-cost criteria are satisfied. Such a screening rule will be weakened/violated under the scenarios when regularization parameter $\gamma$ or the differences between $\hat{w}_j, \hat{w}_{[k]}, \hat{w}_{[k + 1]}$ are small (see Section~\ref{sec:numerical} for detailed scenario settings). To overcome such scenarios, Theorem~\ref{thm:generalized-cuts-generation-rule} proposes the SCG rule that greatly enhances the reduced-cost criteria in Proposition~\ref{prop:AA-screen-rule} by taking into account a specific subset of binaries, other than a single binary variable.

Before presenting detailed results in Theorem~\ref{thm:generalized-cuts-generation-rule}, let us start with the definition of SCG-tuple as follows: Recall $\hat{\bm{\beta}}$ the optimal solution of \eqref{eq:conic-relaxation}, $\hat{\bm{p}}$ the Fenchel dual variable obtained from \eqref{eq:fenchel-relaxation} given $\hat{\bm{\beta}}$. Let $v_{\tt ub}$ and $v_{\tt conic}$ be an upper and lower (optimal value from \eqref{eq:fenchel-relaxation}) bounds of \eqref{eq:original-formulation}, respectively.   

\begin{definition}[\textbf{SCG-tuple $(S, N, C)$}] \label{def:SCG-tuple}
    Given two disjoint index subsets $S, N \subseteq [\! [d]\!]$ with $|S| \leq k$, let $R$ be the remaining set of their union, i.e. $R := [\! [d]\!] \setminus (S \cup N)$. For any index set $C \subseteq R$, we say $(S,N,C)$ is a \emph{SCG-tuple} with respect to $\hat{\bm{w}} = \hat{\bm{p}} \circ \hat{\bm{p}}$, if index set $C$ is constructed as follows: $$C := \left\{ i \in R ~|~ \hat{w}_i \geq (\hat{w}_R)_{[c]} \right\}$$ with $c := \min \{k - |S|, ~ |R|\}$. Specifically, when $c = 0$, we set $C := \emptyset$. In the later of this paper, for a SCG-tuple $(S, N, C)$, we call index set $S$ as the \emph{candidate set of support}, index set $N$ as the \emph{candidate set of non-support}, and index set $C$ as the  \emph{complement set of support} to $S$, which is established based on $S, N$ and $\hat{\bm{w}}$. 
\end{definition}
Intuitively, given a SCG-tuple $(S, N, C)$, set $S$ represents the set of indices that are believed to become support, $N$ represents the set of indices that are believed not to become support, while set $C$ is fully determined by sets $S$ and $R$ as the ``best'' possible set based on $\hat{\bm{w}}$ that, combining with $S$, makes-up a $k$-sparse support set for \eqref{eq:original-formulation}. We are poised to present the theorem for Screening Cut Generation (SCG).

\begin{restatable}[\textbf{Screening Cuts Generation}]{theorem}{optcutgen}
\label{thm:generalized-cuts-generation-rule} 
Suppose components in $\hat{\bm{w}} = \hat{\bm{p}} \circ \hat{\bm{p}}$ are distinct. Given a SCG-tuple $(S, N, C)$ with respect to $\hat{\bm{w}}$, if the inequality of reduced-costs 
\begin{equation}
    \label{ineq:cuts-condition}
    \sum_{i=1}^k \hat{w}_{[i]} - \sum_{i \in S} \hat{w}_i - \sum_{i \in C} \hat{w}_i > \frac{4}{\gamma} \gap 
\end{equation} 
holds, then the following screening cut $$\sum_{i \in S}z_i + \sum_{i \in N} (1 - z_i) \leq |S| + |N| - 1$$
eliminates some non-optimal but remains all optimal solutions for \eqref{eq:original-formulation}. Additionally, we call a SCG-tuple $(S,N,C)$ a \emph{valid SCG-tuple} if $(S,N,C)$ also ensures the above inequality~\eqref{ineq:cuts-condition}. 
\end{restatable}

The proof of Theorem~\ref{thm:generalized-cuts-generation-rule} is presented in Appendix~\ref{app:prop-generalized-cuts-generation-rule}.One can relax the assumption ``components in $\hat{\bm{w}} = \hat{\bm{p}} \circ \hat{\bm{p}}$ are distinct'' by breaking ties lexicographically. To simplify our analysis, in the latter of this paper, \emph{we always assume that there is no tie in $\hat{\bm{w}}$ without special instructions.}   

To be concise, Theorem~\ref{thm:generalized-cuts-generation-rule} provides a generalized screening rule that identifies whether a specific $\{\pm 1\}$ combination of binaries, i.e., $z_i = 1$ for $i \in S$ and $z_i = 0$ for $i \in N$, is in one of optimal solutions or not for the original sparse ridge regression \eqref{eq:original-formulation}. The validness of the corresponding SCG-tuple $(S, N, C)$ is determined by the inequality \eqref{ineq:cuts-condition}, which is a linear underestimate of Lagrangian relaxation constrained by candidate sets $S$ and $N$ given $\hat{\bm{p}}$, 
\begin{align*}
     \begin{array}{rlll}
        \max_{\bm{p}}\min_{\bm{\beta}, \bm{z}} & \mathcal{L}(\bm{\beta}) + \gamma \sum_{i=1}^d ({p}_i\beta_i - \frac{{p}_i^2}{4}z_i) \\
        \text{s.t.} & \bm{1}^{\top} \bm{z} \leq k, ~ \bm{z} \in [0, 1]^d \\
        & z_i = 1 ~ \forall ~ i \in S, ~~ z_i = 0 ~ \forall ~ i \in N
    \end{array} 
    \geq & ~ v_{\tt conic} + \frac{\gamma}{4}\left(\sum_{i=1}^k \hat{w}_{[i]}  - \sum_{i \in S} \hat{w}_i - \sum_{i \in C} \hat{w}_i \right) > v_{\tt ub} \\
    \Leftrightarrow & ~~ \text{inequality~\eqref{ineq:cuts-condition}} ~~. 
\end{align*}
That is to say, if inequality~\eqref{ineq:cuts-condition} holds, the corresponding choice of support subset $S$ and non-support subset $N$ will not be included in any of the optimal solutions of (\ref{eq:original-formulation}). To further enhance the validity or effectiveness of inequality~\eqref{ineq:cuts-condition}, in addition to increasing the term on the left-hand-side by taking multiple binaries into account, it is reasonable to reduce the right-hand-side term $4 \gap / \gamma$. Notably, Theorem~\ref{thm:generalized-cuts-generation-rule} does not specify which upper bounds $v_{\tt ub}$ of \eqref{eq:original-formulation} are used in SCG. Indeed, any upper bounds are allowed, and \textit{a tighter relaxation or a method for improved primal feasible solutions will lead to a smaller gap, thereby enhancing the effectiveness of the proposed screening criterion.} However, in some real applications, the ridge regularization parameter $\gamma$ lies in a relatively low regime, which weakens the effectiveness of Inequality~\eqref{ineq:cuts-condition}. One possible direction (also discussed in Section~\ref{sec:conclusion}) to handle low/zero $\gamma$ regime is to study the following sequence of optimal solutions $\{\hat{\boldsymbol{\beta}}_i\}_{i \geq 1}$ and their support sets obtained from a sequence of auxiliary sparse ridge regressions 
\begin{align*}
    \hat{\boldsymbol{\beta}}_i := \argmin_{\bm{\beta}} ~ \mathcal{L}(\boldsymbol{\beta}) + \gamma_i \|\boldsymbol{\beta}\|_2^2 ~~ \text{s.t.} ~~ \|\boldsymbol{\beta}\|_0 \leq k ~~, 
\end{align*}
where $\{\gamma_i\}_{i \geq 1}$ is a sequence of monotonically decreasing positive reals that tends to zero. Some preliminary numerical results (will be investigated as future directions, also mentioned in Section~\ref{sec:conclusion}) indicate that in many cases, the supports of the solutions $\{\hat{\boldsymbol{\beta}}_i\}_{i \geq 1}$ remain unchanged for sufficiently small $\gamma_i$. \textit{This stability on support suggests that the screening cuts\footnote{These screening cuts may not be the safe screening cuts that only eliminates non-optimal solutions.} computed from sparse ridge regression with small $\gamma_i$ could extend to the general convex sparse regression setting.}

\begin{remark} \label{rmk:AA_rule}
Compared to existing screening rules (Proposition~\ref{prop:AA-screen-rule}), since the left-hand-side term of inequality \eqref{ineq:cuts-condition} involves multiple binaries other than a single binary, the reduced-cost criteria for SCG are more likely to ensure the validity than the reduced-cost criteria in Proposition~\ref{prop:AA-screen-rule}. Moreover, let $\texttt{Top}_k(\hat{\bm{w}}):=\{i \in [\![d]\!] ~|~\hat{w}_i \geq \hat{w}_{[k]} \}$ be the set of indices for the top-$k$ components. By setting SCG-tuple as 

\begin{align*}
    (S, N, C) = & ~\left\{
    \begin{array}{lll}
        \left( \{i\}, ~ \emptyset, ~ \texttt{Top}_k(\hat{\bm{w}}) \setminus \{i\} \right) & ~\text{if}~~~ i \in \texttt{Top}_k(\hat{\bm{w}}) \\
        \left( \{i\}, ~ \emptyset, ~ \texttt{Top}_{k-1}(\hat{\bm{w}}) \right) & ~ \text{otherwise}
    \end{array}
    \right. \\
    ~~\text{or} ~ & ~\left\{
    \begin{array}{lll}
        \left( \emptyset, ~ \{i\}, ~  \texttt{Top}_{k+1}(\hat{\bm{w}}) \setminus \{i\} \right) & ~\text{if}~~~ i \in \texttt{Top}_k(\hat{\bm{w}})\\
        \left( \emptyset, ~ \{i\}, ~  \texttt{Top}_{k}(\hat{\bm{w}}) \right) & ~ \text{otherwise}
    \end{array}
    \right. ~~. 
\end{align*}
The proposed SCG criteria reduce to existing criteria in Proposition~\ref{prop:AA-screen-rule}. The proof of this reduction is given in Appendix~\ref{app:AA_rule}.  
\end{remark}

Additionally, a corollary (Corollary~\ref{coro:valid-tuple}) of Theorem~\ref{thm:generalized-cuts-generation-rule} is presented in Appendix~\ref{app:valid-tuple}, which gives an equivalent algebraic (and more direct) approach to verify whether a given SCG-tuple is valid or not. The Corollary~\ref{coro:valid-tuple} will be used later in the proof of Proposition~\ref{prop:minimal-tuple}.

\subsection{Minimal SCG-Tuple \& Dominating Screening Cuts} \label{sec:SCG-property}

This subsection introduces minimal SCG-tuple and corresponding dominating screening cuts. 

\begin{restatable}{proposition}{cutdomination} \label{prop:cut-domination}
Consider two valid SCG-tuples $(S_1, N_1, C_1), (S_2, N_2, C_2)$, if $S_1 \subseteq S_2$ and $N_1 \subseteq N_2$, then the screening cut generated by $(S_2, N_2, C_2)$ is dominated by the screening cut generated by $(S_1, N_1, C_1)$.
\end{restatable}
The proof of Proposition~\ref{prop:cut-domination} is given in Appendix~\ref{app:cut-domination}. Note that the above dominating/dominated relationship is established based on the inclusion of index sets $S$ and $N$. To avoid dominated cuts and improve computational efficiency in numerical experiments, we first define the following minimal SCG-tuple (from the perspective of set inclusion). 

\begin{definition}
    \label{def:minimal-tuple}
    We call a valid SCG-tuple $(S,N,C)$ with respect to $\hat{\bm{w}}$ as a \emph{minimal SCG-tuple}, if the following two conditions hold simultaneously: 
    \begin{enumerate}
        \item[(1)] $N = \emptyset$, or $(S,N \setminus \{i\},C)$ cannot form a valid SCG-tuple for any $i \in N$,
        \item [(2)] $S = \emptyset$, or $(S \setminus \{i\}, N, C\cup\{i\})$ cannot form a valid SCG-tuple for any $i \in S$.
    \end{enumerate}
\end{definition}

Next, we give necessary and sufficient conditions that ensure a valid SCG-tuple is minimal.  

\begin{restatable}{proposition}{minimal}
    \label{prop:minimal-tuple}
Suppose components in $\hat{\bm{w}} = \hat{\bm{p}} \circ \hat{\bm{p}}$ are distinct. For a valid SCG-tuple $(S, N, C)$ with respect to $\hat{\bm{w}}$, we have:
\begin{itemize}
    \item \textbf{Case 1.} If $|N| > d - k$, then $(S,N,C)$ is minimal if and only if $S = \emptyset$.
    \item \textbf{Case 2.} If $|N| \leq d - k$ and $C = \emptyset$, then $(S,N,C)$ is minimal if and only if $N = \emptyset$ and $\operatorname{max}_{i \in S} ~\hat{w}_i < \hat{w}_{[1]}$. 
    \item \textbf{Case 3.} If $|N| \leq d - k$ and $C \neq \emptyset$, then $(S,N,C)$ is minimal if and only if (1) \& (2) hold simultaneously, 
    \begin{enumerate}
        \item [(1)]  $N = \emptyset$, or $\underline{\hat{{w}}_{N}} \geq \underline{\hat{{w}}_{C}}$;
        \item [(2)] $S = \emptyset$, or there exists some index $j \notin  (S \cup N \cup C)$ such that $\underline{\hat{w}_{C}} >\hat{w}_j > \max_{i \in S} ~ \hat{w}_i$.
    \end{enumerate}
\end{itemize}
\end{restatable}

The proof of Proposition \ref{prop:minimal-tuple} is given in Appendix \ref{app:minimal-tuple}. The results proposed in Proposition~\ref{prop:minimal-tuple} will be later used to construct screening cuts with relatively strong ``potential screening ability'' in Section \ref{sec:implementation-technique}. One quick example is that the valid SCG-tuple $(\{i\},\emptyset,\texttt{Top}_{k-1}(\hat{\bm{w}}))$ guaranteed in Remark \ref{rmk:AA_rule} with $i \notin \texttt{Top}_k(\hat{\bm{w}})$ is minimal. It can be shown by verifying: (1) $N = \emptyset$; (2) the index $i^*$ with respect to $k$-th largest component of $\hat{\bm{w}}$, i.e., $\hat{w}_{i^*} = \hat{w}_{[k]}$, satisfies $i^* \notin \left(\texttt{Top}_{k-1}(\hat{\bm{w}}) \cup \{i\} \right)$ and $\hat{w}_{[k-1]} > \hat{w}_{i^*} > \hat{w}_i$.

\subsection{Connections With $k$-Cardinality Binary Knapsack Polytope} \label{sec:k-knapsack-polytope}

Recall the inequality~\eqref{ineq:cuts-condition} proposed in Theorem~\ref{thm:generalized-cuts-generation-rule} can be represented as: $\sum_{i \in S}\hat{w}_i + \sum_{i \in C}\hat{w}_i  < - \frac{4}{\gamma} \gap + \sum_{i=1}^k \hat{w}_{[i]} =: c(\hat{\bm{w}})$. Consider the following $k$-cardinality binary knapsack polytope $K(\hat{\bm{w}})$: 
\begin{equation}
    \label{eq:knapsack-set}
    K(\hat{\bm{w}}):=\left\{ \bm{x} \in \{0,1\}^d~\Bigg{|}~ \sum_{i =1}^d\hat{w}_i x_i < c(\hat{\bm{w}}),~ \|\bm{x}\|_0 \leq k \right\} 
\end{equation}
It is easy to conclude that: On one side, every valid SCG-tuple $(S, N, C)$ (with respect to $\hat{\bm{w}}$) corresponds to one feasible point $\bm{x} \in K(\hat{\bm{w}})$ by letting  $\supp(\bm{x}) = S \cup C$. On the other side, if $K(\hat{\bm{w}}) \neq \emptyset$ with $c(\hat{\bm{w}}) > 0$, every feasible point $\bm{x} \in K(\hat{\bm{w}})$ corresponds to at least $2^{|\supp(\bm{x})|}$ valid SCG-tuples $(S, N, C)$ by setting $N = [\! [d]\!] \setminus \supp(\bm{x})$ and disjoint subsets $S, C$ such that $S \cup C = \supp(\bm{x})$. 

To simplify the analysis in this subsection, \emph{we assume that values of components in $\hat{\bm{w}}$ are in decreasing order, i.e., $\hat{w}_1 > \hat{w}_2 > \cdots > \hat{w}_d$.} Otherwise, we can define a new vector $\tilde{\bm{w}} := \bm{E} \hat{\bm{w}}$ with $\bm{E}$ a permutation matrix that re-orders every component of $\hat{\bm{w}}$ in a decreasing order. In the latter part of the section, we use $\tilde{\bm{w}}$ to denote the re-ordering of $\hat{\bm{w}}$ for simplicity.

Based on previous discussions, we note that every feasible $\bm{y} \in K(\tilde{\bm{w}})$ leads to multiple valid SCG-tuples. However, as hinted from the Proposition~\ref{prop:minimal-tuple}, we are only interested in minimal SCG-tuples for dominating screening cuts. To figure out the connections between any feasible vector $\bm{x} \in K(\tilde{\bm{w}})$ and minimal SCG-tuples, we give the following definition.

\begin{definition}[\textbf{Consecutive index set}] \label{def:consecutive-index-set} 
A non-empty index set $A \subseteq [\![d]\!]$ is called a \emph{consecutive index set} if, for every index $i \in A$, one of the following three cases holds: $\{i\} = A$, $i+1 \in A$, or $i-1 \in A$.
\end{definition} 

Note that for a given index set $I$, there are usually multiple ways to partition $I$ into several disjoint consecutive index sets. For example, one can partition index set $I = \{2,3,5,6,9\}$ into 
\begin{equation*}
    I =  ~~ \{2,3\} \cup \{5,6\} \cup \{9\} \quad
     \text{or} \quad \{2\} \cup \{3\} \cup \{5,6\} \cup \{9\} \quad
    \text{or} \quad \{2\}\cup \{3\}\cup\{5\}\cup \{6\}\cup \{9\} ~~. 
\end{equation*}
Later in the proof, we use \emph{minimum consecutive partition} to denote the approach that partite a given index set into disjoint consecutive index sets with the minimum number of partitions. For example, the minimum consecutive partition of $I$ is $\{2,3\} \cup \{5,6\} \cup \{9\}$. Moreover, based on our assumption that values of components in $\tilde{\bm{w}}$ are in decreasing order, the complementary index set can be defined as follows: 

\begin{definition}[\textbf{Complementary index set}] 
Given a non-empty index subset $A \subseteq [\![d]\!]$, define \emph{complementary index set} of $A$ as $\comp(A):= \left \lbrace i \in [\![d]\!]~|~i \notin A ~\text{and}~ \exists ~ j \in A ~\text{such that}~ i < j \right\rbrace.$ 
\end{definition}

Based on Defintion~\ref{def:consecutive-index-set}, for every $\bm{0} \neq \bm{y} \in K(\Tilde{\bm{w}})$, the minimum consecutive partition of $\supp(\bm{y})$ can be represented as 
\begin{align*}
    \supp(\bm{y}) = \cup_{j=1}^m A_j(\bm{y}) ~~ \text{for some $m \leq k$ with} ~~ A_j (\bm{y}) := \left\{i^{(j)}_1, \ldots, i^{(j)}_{|A_j|} \right\} \subseteq \supp(\bm{y})
\end{align*}
for all $j  = 1, \ldots, m$. Then, the minimum consecutive partition of $\comp(\supp(\bm{y}))$ can be partitioned into at most $m$ disjoint blocks:
\begin{align*}
    \comp(\supp(\bm{y})) = \cup_{j=1}^m B_j(\bm{y}) ~~\text{with}~~B_j(\bm{y}) := \left\{i^{(j - 1)}_{|A_{j - 1}|} + 1, \ldots, i^{(j)}_1 - 1 \right\} 
\end{align*}
for all $j = 1, \ldots, m$, where we let $i^{(0)}_{|A_{0}|} = 0$. Here, $B_1(\bm{y})$ might be empty. Now, we are ready to characterize the minimal SCG-tuple $(S, N, C)$ based on any non-trivial $\bm{y} \in K(\tilde{\bm{w}})$. 

\begin{restatable}{proposition}{knapsackcuts}
    \label{prop:knapsack-cuts-generation}
Suppose components in $\tilde{\bm{w}}$ are distinct. Given any non-trivial $\bm{y} \in K(\tilde{\bm{w}})$. Let the minimum consecutive partitions of $\supp(\bm{y})$ and $\comp(\supp(\bm{y}))$ be 
\begin{equation*}
     ~ \supp(\bm{y}) = \cup_{j = 1}^m A_j(\bm{y}) \quad 
    \text{and} \quad  \comp(\supp(\bm{y})) = \cup_{j = 1}^m B_j(\bm{y}) 
\end{equation*}
with some $m \leq k$. Then, one can construct minimal SCG-tuple $(S, N, C)$ with $S \cup C = \supp (\bm{y})$ as follows:
\begin{enumerate}
    \item[(1)]If $|\supp(\bm{y})| < k$, then $$(S, N, C) = (\emptyset, ~ [\![d]\!] \setminus \supp(\bm{y}), ~ \supp(\bm{y}) )$$ is a minimal SCG-tuple. 

    \item[(2)] If $|\supp(\bm{y})| = k$, for $\ell = 1, \ldots, m - 1$, we have the following SCG-tuple
    \begin{align*}
        (S, N, C) = \left( \cup_{j = \ell + 1}^m A_j (\bm{y}), ~ \cup_{j = 1}^{\ell} B_j (\bm{y}), ~ \cup_{j = 1}^{\ell} A_j (\bm{y}) \right)
    \end{align*}
    is a minimal SCG-tuple. For $\ell = 0$ or $\ell = m$, define
    \begin{equation*}
         ~ (S_1, N_1, C_1) = (\supp(\bm{y}), ~ \emptyset, ~ \emptyset) \quad
        \text{or} \quad (S_2, N_2, C_2) = (\emptyset, ~ \comp(\supp(\bm{y})), ~ \supp(\bm{y})) ~~, 
    \end{equation*}
    respectively. Then $(S_2,N_2,C_2)$ is minimal. Moreover, if $B_1(\bm{y}) \neq \emptyset$, $(S_1,N_1,C_1)$ is minimal as well.
\end{enumerate}

\end{restatable}
The proof of Proposition~\ref{prop:knapsack-cuts-generation} is presented in Appendix~\ref{app:knapsack-cuts-generation}. Based on the construction procedure given in Proposition~\ref{prop:knapsack-cuts-generation}, one can generate at most $m + 1$ minimal SCG-tuples (and corresponding screening cuts) from any non-trivial point $\bm{y} \in K(\tilde{\bm{w}})$.

\subsection{Cuts With Strong Screening Ability} \label{sec:implementation-technique}

However, from numerical results reported in Section~\ref{sec:numerical}, it is relatively inefficient (in practice) to incorporate all possible screening cuts (obtained by minimal SCG-tuples) into the original sparse ridge regression. To enhance computational efficiency, this subsection selects screening cuts (obtained by minimal SCG-tuples) with relatively strong ``\emph{potential screening ability},'' a criterion that quantifies the largest possible number of feasible solutions eliminated by corresponding screening cuts. 
\begin{definition} \label{def:potential-screening-ability}
Given a valid SCG-tuple $(S, N, C)$, sorted vector $\tilde{\bm{w}}$ given in Section~\ref{sec:k-knapsack-polytope}, and corresponding screening cut $\sum_{i \in S} z_i + \sum_{i \in N} (1 - z_i) \leq |S| + |N| - 1$, its ``\emph{potential screening ability}'' is $\sum_{i = 0}^ {|C|} {d- |S \cup N| \choose i}$, which exactly denotes the number of feasible solutions that is eliminated by incorporating the above screening cuts into the original sparse ridge regression. 
\end{definition}
The proof of bound given in Definition~\ref{def:potential-screening-ability} is presented in Appendix~\ref{app:potential-screening-ability}. 

It is easy to observe that a smaller cardinality on set $S \cup N$ and a larger cardinality on set $C$ will lead to a relatively greater ``potential screening ability'' for a given valid SCG-tuple $(S, N, C)$. Note that, in practice, a set of screening cuts are selected to incorporate into the original sparse ridge regression. Thus, the additional number of non-optimal solutions removed by a newly added screening cut heavily depends on the rest of the selected screening cuts, and may be smaller than the ``potential screening ability'' provided above. We refer to the above number of eliminated non-optimal solutions as ``\emph{(marginal) screening ability}'', which may be difficult to compute or estimate in practice. \\

\noindent\textbf{High-level idea of SCG algorithm:} To further enhance the computational efficiency for BnB, the proposed algorithm (Algorithm~\ref{alg:alg-framework}) makes a trade-off between the number of screening cuts selected and their corresponding ``screening ability'' based on the following three high-level criterion (decreasing in importance): (1) first, selecting cuts with relatively higher potential screening ability; (2) second, selecting cuts with relatively higher (marginal) screening ability; (3) third, if two cuts enjoy the same screening ability or potential screening ability, selecting the cut with greater left-hand-side value of the inequality in \eqref{eq:knapsack-set}. Thus, this subsection focuses on selecting minimal SCG-tuple (dominating screening cuts) of the following two types. \\

\noindent \textbf{Inclusive cuts.} 
The first type of dominating screening cuts takes the form: $\sum_{i \in N} z_i \geq 1$. In short, such cuts indicate that any optimal solutions of (\ref{eq:original-formulation}) contain at least one support index from $N$. According to Proposition \ref{prop:knapsack-cuts-generation}, for any $\bm{y} \in K(\Tilde{\bm{w}})$, by setting $S = \emptyset$ and $C = \supp (\bm{y})$, the resulting SCG-tuple should be
\begin{align*}
    (S, N, C) = \left\{
    \begin{array}{lll}
        (\emptyset, ~ [\![d]\!] \setminus \supp(\bm{y}), ~ \supp(\bm{y})) & \text{ if } ~~ |\supp(\bm{y})| < k \\
        (\emptyset, ~ \comp(\supp(\bm{y})), ~ \supp(\bm{y})) & \text{ if }  ~~ |\supp(\bm{y})| = k 
    \end{array}
    \right.
\end{align*}
to ensure its validity and minimality. Note that when $|\supp(\bm{y})| < k$, the cardinality of set $N$ in above SCG-tuple $(S, N, C)$ is $d - |\supp(\bm{y})|$, which greatly impedes the ``potential screening ability''. Thus, inclusive cuts generation method only focuses on finding points in the following subset $\mathcal{O}:=\{\bm{y} \in K(\tilde{\bm{w}})~\big|~|\texttt{supp}(\bm{y})| = k\}$ with minimal tuples $(S,N,C) = (\emptyset,\comp(\supp(\bm{y})),\supp(y))$. To simplify algorithm implementation (as will be mentioned in subSection~\ref{sec:algorithm-design}), partition $\mathcal{O}$ into disjoint subsets based on the last index of $\supp(\bm{y})$ as $\mathcal{O} = \cup_{i=1}^d \mathcal{T}_i$ such that 

\begin{equation} \label{eq:inclusive-cut}
    \mathcal{T}_i := \{ \bm{y} \in K(\Tilde{\bm{w}})~\big|~|\supp(\bm{y})| = k, \supp(\bm{y}) \subseteq [\![i]\!], i \in \supp(\bm{y})\}
\end{equation}
a subset of $\mathcal{O}$ with $i$ as the last index of $\supp(\bm{y})$ for all $i \in [\![d]\!]$. Given components of $\Tilde{\bm{w}}$ are in decreasing order, we have $\mathcal{O} = \cup_{i= k + 1}^d \mathcal{T}_i$ since $\mathcal{T}_i = \emptyset$ for all $i = 1, \ldots, k$. In particular, $\mathcal{T}_k = \emptyset$ holds because, if not, for any $\bm{y}' \in \mathcal{T}_k$ with $\texttt{supp}(\bm{y}') = [\![k]\!]$, we have $\sum_{i =1}^d \tilde{w}_i y'_i = \sum_{i=1}^k \tilde{w}_i \geq c(\tilde{\bm{w}}) =  - \frac{4}{\gamma} \texttt{gap} + \sum_{i=1}^k \tilde{w}_i$, which leads to a contradiction. Then we present the Claim~\ref{clm:inclusive}, which is used in Algorithm~\ref{alg:alg-framework}.

\begin{restatable}{claim}{inclusive}
    \label{clm:inclusive}
    Given $\mathcal{T}_i$ with $i \in [\![k+1,d]\!]$, $\mathcal{T}_i \neq \emptyset$ if and only if there exists some $\bm{v} \in K(\Tilde{\bm{w}})$ such that $\supp(\bm{v}) = [\![i - (k-1),i]\!]$. Moreover, if $\mathcal{T}_{i} \neq \emptyset$, we have $\mathcal{T}_j \neq \emptyset$ for all $j = i+1,\ldots,d$.
\end{restatable}
The proof of Claim~\ref{clm:inclusive} is given in Appendix \ref{app:algorithm}. In short, Claim~\ref{clm:inclusive} provides a necessary and sufficient condition to verify whether $\mathcal{T}_i$ is empty or not. Notably, finding the smallest index $s \in [\![k + 1, d]\!]$ such that $\mathcal{T}_s \neq \emptyset$ is equivalent to finding all non-empty $\mathcal{T}_i$s, such a property is also used in Algorithm~\ref{alg:alg-framework}. \\

\noindent \textbf{Exclusive cuts.} The second type of screening cuts takes the form $\sum_{i \in S} z_i \leq |S| - 1$, which ensures that not all indices from $S$ are support indices for any optimal solutions of \eqref{eq:original-formulation}. Recall the minimal consecutive partitions of $\supp(\bm{y})$ and $\comp(\supp(\bm{y}))$ are $\supp(\bm{y}) = \cup_{j = 1}^m A_j(\bm{y})$ and  $\comp(\supp(\bm{y})) = \cup_{j = 1}^m B_j(\bm{y})$. According to Proposition \ref{prop:knapsack-cuts-generation}, for any $\bm{y} \in K(\Tilde{\bm{w}})$ with $|\supp(\bm{y})| = k$, the resulting SCG-tuple should be 

\begin{align*}
    (S, N, C) = \left\{
    \begin{array}{lll}
        (\supp(\bm{y}), ~ \emptyset, ~ \emptyset) 
        & \text{ if } ~~ B_1({\bm{y}}) \neq \emptyset \\
        (\cup_{j=2}^m A_j(\bm{y}), ~ \emptyset, ~ A_1(\bm{y})) 
        & \text{ if } ~~ B_1({\bm{y}}) = \emptyset 
    \end{array}
    \right. ~~
\end{align*}
to ensure its validity and minimality. Note that when $B_1(\bm{y}) \neq \emptyset$, the SCG-tuple $(\supp(\bm{y}),\emptyset,\emptyset)$ only eliminates one single point based on the ``potential screening ability''. Therefore, exclusive cuts generation method focuses on finding solutions in the following subset 
\begin{align*}
    \mathcal{O}':= \{\bm{y} \in K(\Tilde{\bm{w}})~\big|~ |\supp(\bm{y})| = k,~B_1(\bm{y}) = \emptyset \} = \{\bm{y} \in K(\Tilde{\bm{w}})~\big|~ |\supp(\bm{y})| = k,~1 \in A_1(\bm{y}) \}
\end{align*}
with minimal tuples $(\cup_{j=2}^m A_j(\bm{y}), ~ \emptyset, ~ A_1(\bm{y}))$. Similar to inclusive cuts generation, partition $\mathcal{O}'$ into disjoint subsets based on $A_1(\bm{y})$ as $\mathcal{O}' = \cup_{i=1}^d \mathcal{T}_i'$ such that
\begin{equation} \label{eq:exclusive-cut}
    \mathcal{T}_i':= \{ \bm{y} \in K(\Tilde{\bm{w}})~\big|~|\supp(\bm{y})| = k, A_1(\bm{y}) = [\![i]\!]  \} 
\end{equation}

Given components of $\Tilde{\bm{w}}$ are in decreasing order, we have $\mathcal{O}' = \cup_{i=1}^{k-1} \mathcal{T}_i'$ since $\mathcal{T}_i' = \emptyset$ for $i = k,\ldots,d$. The following claim provides a necessary and sufficient condition to certify whether each $\mathcal{T}_i'$ is empty or not.

\begin{restatable}{claim}{exclusive}\label{clm:exclusive}
    Given $\mathcal{T}_i'$ with $i \in [\![1,k-1]\!]$, we have $\mathcal{T}_i' \neq \emptyset$ if and only if there exists some $\bm{v} \in K(\Tilde{\bm{w}})$ such that $\supp(\bm{v}) = [\![i]\!] \cup [\![d - (k-i-1),d]\!]$; Furthermore, if $\mathcal{T}_{i}' \neq \emptyset$, we have $\mathcal{T}_j' \neq \emptyset$, for all $j = 1,\ldots,i-1$.
\end{restatable}

The proof of Claim \ref{clm:exclusive} is given in \ref{app:algorithm}. This claim is then used in designing Algorithm \ref{alg:alg-framework}. \\

In summary, the SCG algorithm prefers to select cuts with a relatively stronger potential screening ability than other criteria. For inclusive cut generation, given $\bm{y} \in \mathcal{T}_i \neq \emptyset$, its corresponding inclusive cut involves $|\comp(\supp(\bm{y}))| = i - k$ binary variables. Thus, SCG algorithm prefers to consider $\mathcal{T}_i$ with a smaller index $i ~ (\geq k + 1)$ to ensure a relatively stronger ``potential screening ability''. In contrast, for exclusive cut generation, given $\bm{y}' \in \mathcal{T}_i' \neq \emptyset$, its corresponding exclusive cut involves $|\cup_{j=2}^m A_j(\bm{y})| = k - |A_1(\bm{y})| = k-i$ binary variables. Therefore, SCG algorithm prefers to consider $\mathcal{T}_i'$ with a larger index $i ~ (\leq k-1)$ to ensure a relatively stronger ``potential screening ability''.

\subsubsection{SCG Algorithm Design} \label{sec:algorithm-design}

According to the discussions in Section~\ref{sec:implementation-technique} above, the SCG algorithm is presented in Algorithm \ref{alg:alg-framework}. In particular, the proposed Algorithm \ref{alg:alg-framework} includes two hyper-parameters that control the maximum number of generated cuts $\sharp_{\tt max}$, and the maximum number of binary variables (potential screening ability) involved in each cut $\sharp_{\tt len}$. Additionally, we give brief discussions on two key steps of Algorithm \ref{alg:alg-framework} as follows: \\

\begin{algorithm}[t]
\caption{Inclusive/Exclusive Screening Cuts Generation (SCG)}
\label{alg:alg-framework}
\KwIn{Sorted $\Tilde{\bm{w}}$, sparsity $k$, max number of cuts $\sharp_{\texttt{max}}$, max length of cuts $\sharp_{\texttt{len}}$.}

Set $\texttt{SR} = \texttt{SR}_{\tt inc}$ or $\texttt{SR} = \texttt{SR}_{\tt exc}$ from \eqref{eq:inc-SR-set} or \eqref{eq:exc-SR-set}, set $\texttt{stopflag} = \texttt{false}$\; 
Initialize set of generated cuts: $S_{\texttt{cuts}} = \emptyset$\;
Initialize starting search index $s$ based on (\ref{eq:inc-sr}) or (\ref{eq:exc-sr})\hfill $\triangleright \textcolor{blue}{{\CommentSty{STEP 1}}}$

\SetKw{KwTrue}{\texttt{true}}
\SetKw{StopFlag}{\texttt{stopflag}}

\Repeat{\StopFlag= \KwTrue}{
    Enumerate solutions in $\mathcal{T}_s$ or $\mathcal{T}_s'$ recursively using Algorithm \ref{alg:recursive-iteration}; \hfill $\triangleright \textcolor{blue}{{\CommentSty{STEP 2}}}$

    \For{$\bm{y} \in \mathcal{T}_s$ or $\mathcal{T}_s'$}{
        Add inclusive/exclusive cut into $S_{\texttt{cuts}}$\;
        Update $\sharp_{\texttt{max}} = \sharp_{\texttt{max}} - 1$\;
        \If{$\sharp_{\texttt{max}} \leq 0$}{
            set $\texttt{stopflag} = \texttt{true}$\;
            \textbf{break}\;
        }
    }
    Update search index $s$ based on discussion of STEP 2 in Section~\ref{sec:algorithm-design}\;
    If $s \in \texttt{SR}$, set $\texttt{stopflag} = \texttt{false}$; otherwise set $\texttt{stopflag} = \texttt{true}$\;
    }

\KwOut{Set of generated cuts $S_{\texttt{cuts}}$.}
\end{algorithm}

\noindent\textbf{STEP 1.} 
Given a fixed $\sharp_{\tt len}$ that controls the number of binary variables involved, for inclusive or exclusive cut, it is sufficient to take indices $i$ such that $i \leq k + \sharp_{\tt len}$ or $i \geq k - \sharp_{\tt len}$, respectively. That is to say, SCG algorithm only considers feasible points in $\mathcal{T}_i$s or $\mathcal{T}_i'$s with index $i$ in the corresponding searching range,  
\begin{align}
    & ~ i \in \texttt{SR}_{\tt inc}:=[\![k+1,\min\{k+\sharp_{\tt len},d\}]\!] ~~ \text{for inclusive cuts,} \label{eq:inc-SR-set} \\
    & ~ i \in \texttt{SR}_{\tt exc}:=[\![\max\{1,k - \sharp_{\tt len}\},k-1]\!] ~~ \text{for exclusive cuts}. \label{eq:exc-SR-set}
\end{align}
Therefore, in Step 1, the starting search index $s$ is initialized by  
\begin{align}
    s_{\tt inc} := & ~ \operatorname{argmin} \left\{ i \in \texttt{SR}_{\tt inc}~\big|~\exists~\bm{v} \in K(\Tilde{\bm{w}})~\text{s.t.}~\supp(\bm{v}) = [\![i - (k-1),i]\!]\right\}  \label{eq:inc-sr} \\
    s_{\tt exc} := & ~ \operatorname{argmax} \left\{ i \in \texttt{SR}_{\tt exc}~\big|~\exists~\bm{v} \in K(\Tilde{\bm{w}})~\text{s.t.}~\supp(\bm{v}) = [\![i]\!] \cup [\![d - (k-i-1),d]\!]\right\} \label{eq:exc-sr} 
\end{align}
for inclusive or exclusive cuts, respectively, which denotes the starting search index of $\mathcal{T}_{\cdot}$ or $\mathcal{T}_{\cdot}'$ that Algorithm~\ref{alg:alg-framework} enumerates based on Claim~\ref{clm:inclusive} and Claim~\ref{clm:exclusive}. Additionally, if $\mathcal{T}_i = \emptyset$, $\mathcal{T}_i' = \emptyset$ for all $i$ within the searching range $\texttt{SR}_{\tt inc}, \texttt{SR}_{\tt exc}$, there is no feasible solution in $K(\tilde{\bm{w}})$ that satisfies our needs, and the Algorithm~\ref{alg:alg-framework} terminates. \\

\noindent\textbf{STEP 2.} If current search index $s$ lies in the searching range $\texttt{SR}_{\tt inc}$ or $\texttt{SR}_{\tt exc}$, we run a recursive algorithm (Algorithm \ref{alg:recursive-iteration}) to enumerate solutions in $\mathcal{T}_s$ or $\mathcal{T}_s'$. Note that enumerating all solutions in $\mathcal{T}_s$ or $\mathcal{T}_s'$ with $\sharp_{\tt max} = + \infty$ leads to a worst-case computational complexity $O\big((k-1)^2 \cdot{k + \sharp_{\tt len} - 1 \choose k-1}\big)$ or $O\big((\sharp_{\tt len})^2 \cdot {d \choose \sharp_{\tt len}}\big)$, respectively, which is exponential in $k$ or $\sharp_{\tt len}$. However, in practice, we always set the hyper-parameter $\sharp_{\tt max} \leq O(d)$ and $\sharp_{\tt len} \in \{2,3\}$ (see default hyper-parameter settings in Section~\ref{sec:experimental-setup}), which makes STEP 2 very efficient (takes at most $2$ seconds for synthetic datasets and $30$ seconds for hard real instances, see Table~\ref{tbl:pre-solving-time} for details). Moreover, based on the design of our recursive algorithm, STEP 2 first enumerates solutions with a greater value on the left side of the inequality in \eqref{eq:knapsack-set}, which follows the high-level idea mentioned before.

Then, Algorithm~\ref{alg:alg-framework} updates its current searching index $s$ by adding one for inclusive cuts generation and deleting one for exclusive cuts generation, respectively. Based on Claim \ref{clm:inclusive} and Claim \ref{clm:exclusive}, the updated searching index $s$ ensures $\mathcal{T}_s \neq \emptyset$. The stopping criterion of Algorithm~\ref{alg:alg-framework} is set by: if one of the following two conditions is satisfied: (1) the number of cuts in $S_{\texttt{cuts}}$ exceeds $\sharp_{\tt max}$, (2) search index $s$ beyond its searching range $\texttt{SR}_{\tt inc}$ or $\texttt{SR}_{\tt exc}$. \\

Based on the above analysis, we claim the following result.    
\begin{restatable}{claim}{generateall}\label{clm:generate-all}
By setting the maximal number of inclusive or exclusive cuts $\sharp_{\tt max} = + \infty$, Algorithm \ref{alg:alg-framework} outputs all minimal inclusive cuts with $\sharp_{\tt len}\leq d-k$, or exclusive cuts with maximum length $\sharp_{\tt len} \leq k-1$, respectively. 
\end{restatable}
The proof of Claim~\ref{clm:generate-all} is given in Appendix \ref{app:algorithm}, which ensures that, in theoretical, Algorithm~\ref{alg:alg-framework} outputs all dominating cuts with deserved length (number of binaries involved).

\section{Numerical Experiment} \label{sec:numerical}
This section conducts numerical experiments on both synthetic and real datasets to demonstrate the efficiency of our proposed SCG method compared with the existing baseline, Safe Screening Rules (SSR, \cite{Ata20}), in the pre-processing step.

\subsection{Experimental Setup and Implementation} \label{sec:experimental-setup}

\textbf{Loss function.} 
Numerical experiments are conducted via the following sparse linear ridge regression model: 
\begin{equation} \label{eq:ridge-regression}
    \min_{\|\bm{\beta}\|_0 \leq k} ~ \frac{1}{n} \| \bm{Y} - \bm{X}\bm{\beta}\|_2^2 + \gamma \|\bm{\beta}\|_2^2
\end{equation}
where we use $n$ as the number of samples, $d$ as the dimension of each input sample and decision variable $\bm{\beta}$, coefficient matrix $\bm{X} \in \mathbb{R}^{n \times d}$ is the input sample matrix with its $i$-th row $\bm{X}_{i, :}$ as the $i$-th input sample, coefficient vector $\bm{Y} \in \mathbb{R}^n$ is the output (response) vector with its $i$-th component $\bm{Y}_i$ as the response of $\bm{X}_{i,:}$, $\gamma > 0$ is a pre-determined parameter for $\ell_2$-norm regularization, and $k$ is the sparsity level. We fix the sparsity level $k = 10$ by default. \\

\noindent \textbf{Performance measures.} 
The numerical performance of the proposed SCG and SSR are measured based on the following metrics:
\begin{enumerate}

\item[1.] \emph{Cut Characteristics}: We collect the characteristics or properties of (screening) cuts generated by SCG and SSR. In particular, for SSR, we compute the average number of binaries that can be screened/fixed. For SCG, we check the following two aspects: (i) the average number of generated inclusive and exclusive cuts, and (ii) the average number of binary variables involved in each screening cut.

\item[2.] \emph{Total Running Time}: For SCG and SSR, we measure the total running time (in seconds) required to solve Problem (\ref{eq:ridge-regression}) to same MIPGap, i.e. $\text{MIPGap}:= |v_p - v_d|/|v_p|$, where $v_p$ is the primal objective bound and $v_d$ is the dual objective bound. We use $t_{\tt SCG}$ to denote total running time of SCG, which is given by $t_{\tt SCG} := t^{\tt pre}_{\tt SCG} + t^{\tt sol}_{\tt SCG}$ with $t^{\tt pre}_{\tt SCG}$ its pre-processing time (i.e., time used to select screening cuts by SCG Algorithm~\ref{alg:alg-framework}) and $t^{\tt sol}_{\tt SCG}$ the time used to solving the original problem with added cuts by BnB. Similarly for SSR, its total running time is given by $t_{\tt SSR} := t^{\tt pre}_{\tt SSR} + t^{\tt sol}_{\tt SSR}$.
\end{enumerate}

\noindent \textbf{Testing procedures.} 
For each dataset $(\boldsymbol{X},\boldsymbol{Y})$ with regularization parameter $\gamma$ and default sparsity level $k = 10$, the numerical experiments take the following procedures:
\begin{enumerate}
    \item We first solve the corresponding Problem (\ref{eq:ridge-regression}) using SSR via Gurobi. The Gurobi solver for SSR terminates if one of the following two stopping criteria (SC) is satisfied: (SSR-SC1) the time limit goes to 15 minutes; (SSR-SC2) the MIPGap is less than $1\%$. We use $t^{\texttt{sol}}_{\texttt{SSR}}$ and $\texttt{Gap}_{\texttt{SSR}}$ to denote the running time and MIPGap when Gurobi terminates for SSR method. 

    \item We then solve the corresponding Problem (\ref{eq:ridge-regression}) using the proposed SCG method via Gurobi. The Gurobi solver for SCG terminates if one of the following two stopping criteria is satisfied: (SCG-SC1) the time limit goes to $t^{\texttt{sol}}_{\texttt{SSR}}$ in first bullet; (SCG-SC2) the MIPGap of SCG is less than $\max \{\texttt{Gap}_{\texttt{SSR}} - \varepsilon,1\%\}$ for some small $\varepsilon$ (we set $\varepsilon = 0.1\%$ by default). The (SCG-SC2) is designed to prevent scenarios in which SSR has been stuck at a particular MIP gap value. 
\end{enumerate}

\noindent \textbf{Default hyper-parameter setting for SCG algorithm.} To exceed the maximum number of inclusive cuts ($z_i = 1$ for some $i$) and exclusive cuts ($z_i = 0$ for some $i$) that SSR would select, the hyper-parameter $\sharp_{\tt max}$ in proposed SCG Algorithm~\ref{alg:alg-framework} is set as follows: For inclusive cuts generation, we set max number of inclusive cuts $\sharp_{\tt max} = k = 10$ by default; and for exclusive cuts generation, we set max number of exclusive cuts $\sharp_{\tt max} = d- k = d- 10$ by default. As mentioned above, a larger value on $\sharp_{\tt max}$ gives a tighter outer approximation to the optimal set, while it may impede whole computational efficiency by adding too many constraints. Hence, $\sharp_{\tt max}$ should not be too large in practice. To ensure relatively strong potential screening ability, the hyper-parameter $\sharp_{\tt len}$ is set as $2$ for inclusive cuts generation, and as $3$ for exclusive cuts generation. \\

\noindent \textbf{Hardware $\&$ Software.}  All experiments are conducted in Dell workstation Precision 7920 with a
3GHz 48Cores Intel Xeon CPU and 128GB 2934MHz DDR4 Memory. The proposed method and other baselines are solved using Gurobi 12.0.0 in Python 3.12.8. 

\subsection{Experimental Dataset} \label{sec:description-of-experimental-data}

\noindent \textbf{Synthetic datasets.} We conducted numerical experiments on synthetic datasets with the following settings. 

\begin{itemize}
    \item \textbf{Data generation procedure.} Synthetic datasets $(\bm{X}, \bm{Y})$ are generated as follows (also used in \cite{Ber20}): Each sample $\bm{X}_{i, :}$ is i.i.d. drawn from a $d$-dimensional Gaussian distribution $\mathcal{N}(\boldsymbol{0}, \bm{\Sigma})$ with covariance $\bm{\Sigma}_{ij} = (\rho^{|i-j|})_{i,j}$ for some fixed covariance parameter $\rho \in (0, 1)$ for all $i,j \in [\![d]\!]$. The response vector $\bm{Y}$ is therefore generated by $\bm{Y} = \bm{X}\bm{\beta}^* + \bm{\epsilon}$ for some ground-truth $\bm{\beta}^*$, where $\bm{\beta}^* \in \{-1,0,1\}^d$ is randomly sampled with exactly $k_0$ non-zero entries, and every component of noise vector $\bm{\epsilon}$ is i.i.d. generated from a Gaussian distribution $\mathcal{N}(0, \frac{1}{n} \frac{\|\bm{X} \bm{\beta}^*\|_2^2}{\text{SNR}^2})$ with $\text{SNR}$ denotes the signal-to-noise ratio.

    \item \textbf{Parameter settings.} We set input dimension $d \in \{100,300,500\}$, signal-to-noise ratio $\text{SNR} \in \{0.5, 3.5\}$, and $\gamma \in \{1.4,1.6,1.8,2,2.2\}$. We fix number of samples $n = 50$, covariance parameter $\rho = 0.5$, and number of non-zero entries $k_0 = 10$. For each parameter configuration $(d, \text{SNR},\gamma, k_0 = 10, \rho = 0.5,  n = 50)$, we generate $10$ random independent synthetic datasets, i.e.,  $(\boldsymbol{X}^{(1)}, \boldsymbol{Y}^{(1)}),\ldots, (\boldsymbol{X}^{(10)}, \boldsymbol{Y}^{(10)})$. 

\end{itemize}

\noindent \textbf{Real datasets.} We conducted numerical experiments on two real datasets: CNAE ($856$ features, $1080$ samples) and UJIndoorLoc ($520$ features, $19937$ samples) from the UCI Machine Learning Repository \citep{UCIML}, which are also tested in Section 4.2 in \cite{Ata20}. 

\begin{itemize}
    \item \textbf{Parameter settings.} For each dataset, the sample-dependent parameter $\gamma_0$ for $\ell_2$-norm regularization is set as $\gamma_0:= \frac{d}{k \operatorname{max}_j \| \boldsymbol{X}_{j,:}\|_2^2}$, which follows the setting used in Section 4.2 of \cite{Ata20}. The $\gamma$'s used in our experiments are then ranging from $0.04 / \gamma_0$ to $0.16 / \gamma_0$ under a relatively lower $\gamma$ regime (ranging from 0.04 to 0.16) ignoring the sample-dependent parameter $\gamma_0$. In particular, we set $\gamma \in \{0.04,0.05,0.055,0.06,\ldots,0.09\}$ for dataset CNAE and select $\gamma \in \{ 3000,4000,\ldots,14000\}$ for dataset UJIndoorLoc.
    
    \item \textbf{Reduced sample set.} We further test SSR and SCG on more challenging tasks (in statistics) with reduced sample sets. In particular, we pick the first 200 samples as the reduced sample set for CNAE, and pick 300 samples (first 60 samples among five distinct classes) as the reduced sample set for UJIndoorLoc. 
\end{itemize}

\begin{figure*}[t!]
\vskip 0.2in
\centering
\begin{subfigure}
\centering
\includegraphics[width=0.32\linewidth]{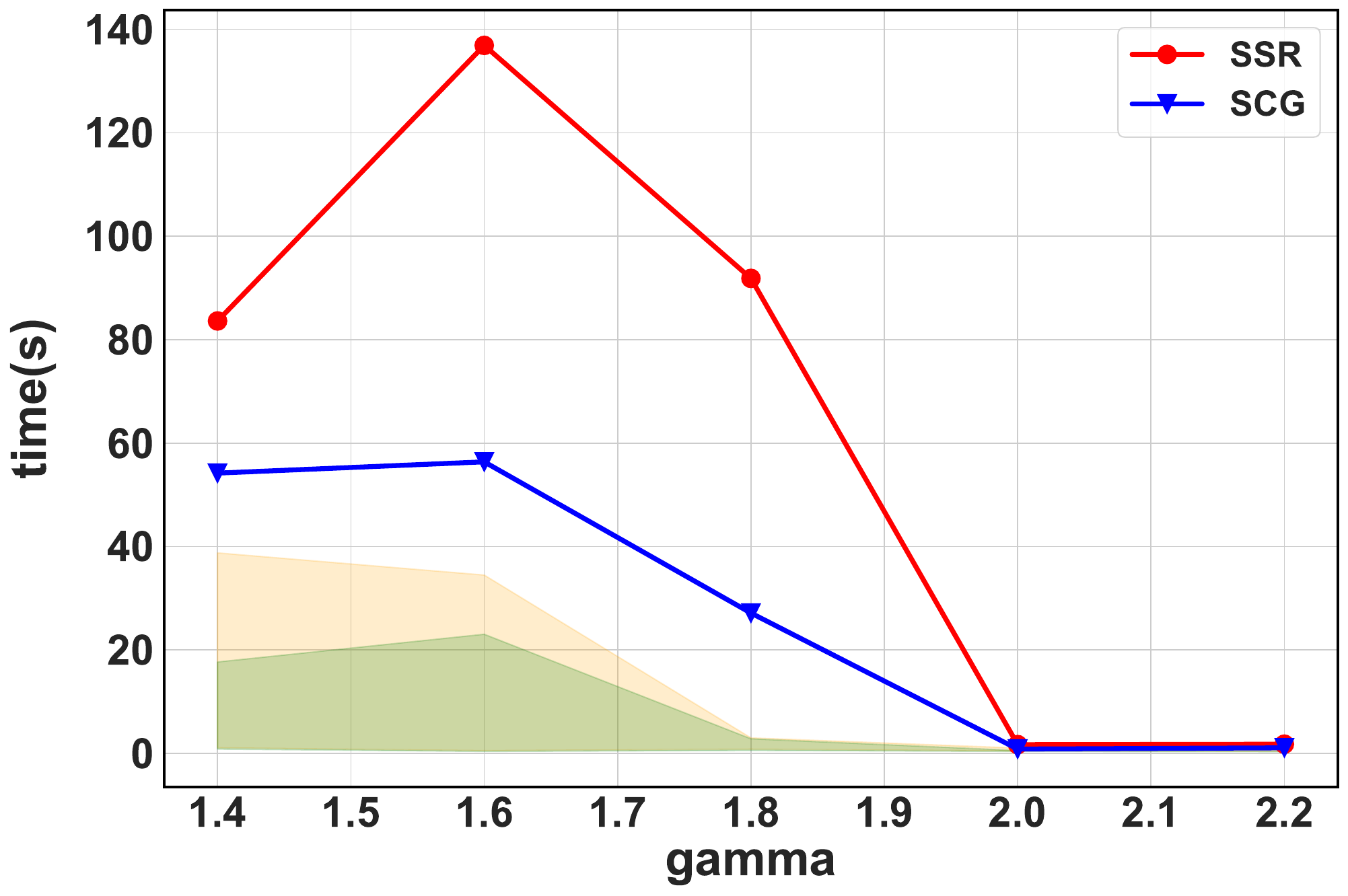}
\end{subfigure}
\begin{subfigure}
\centering
\includegraphics[width=0.32\linewidth]{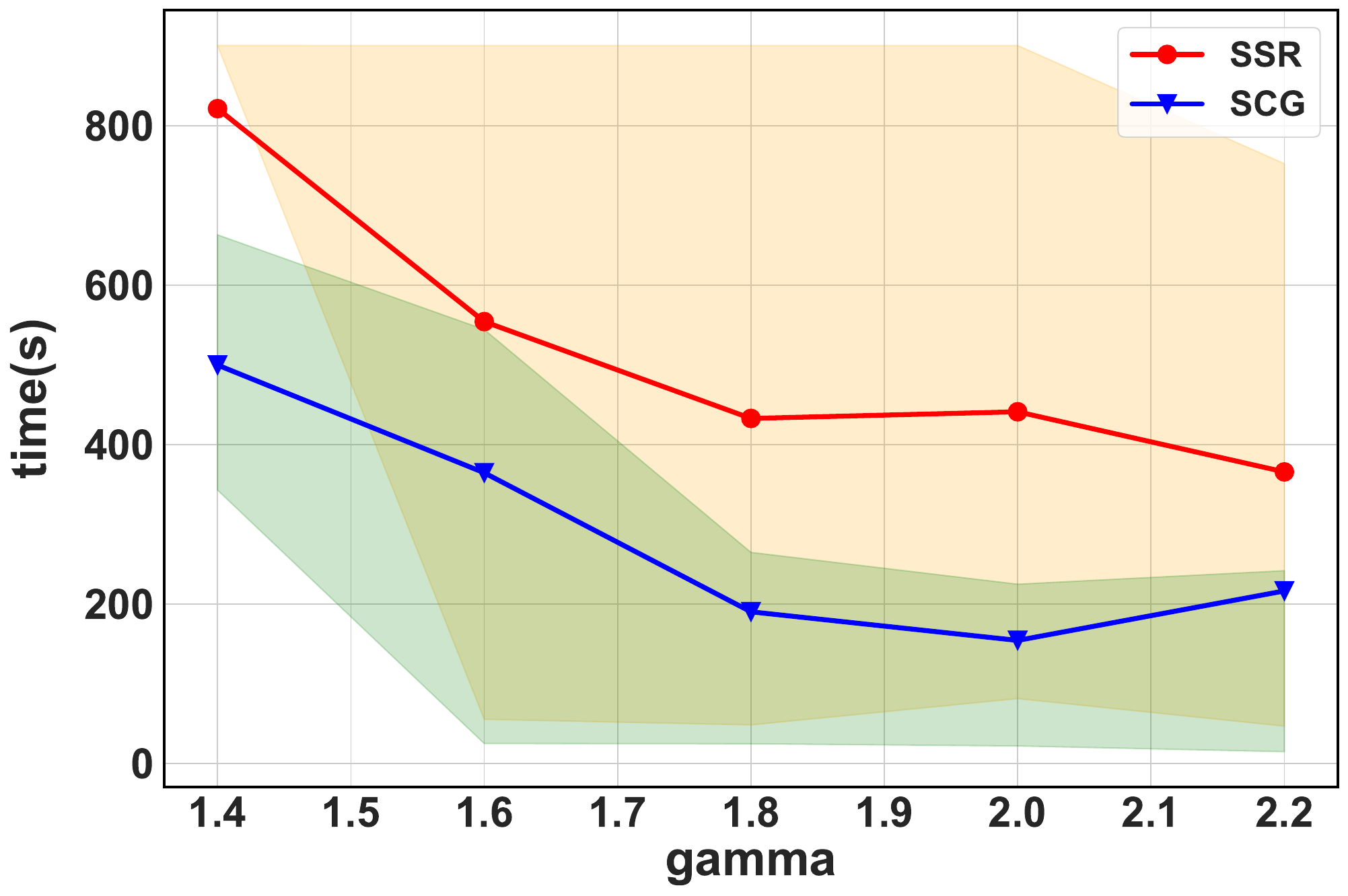}
\end{subfigure}
\begin{subfigure}
\centering
\includegraphics[width=0.32\linewidth]{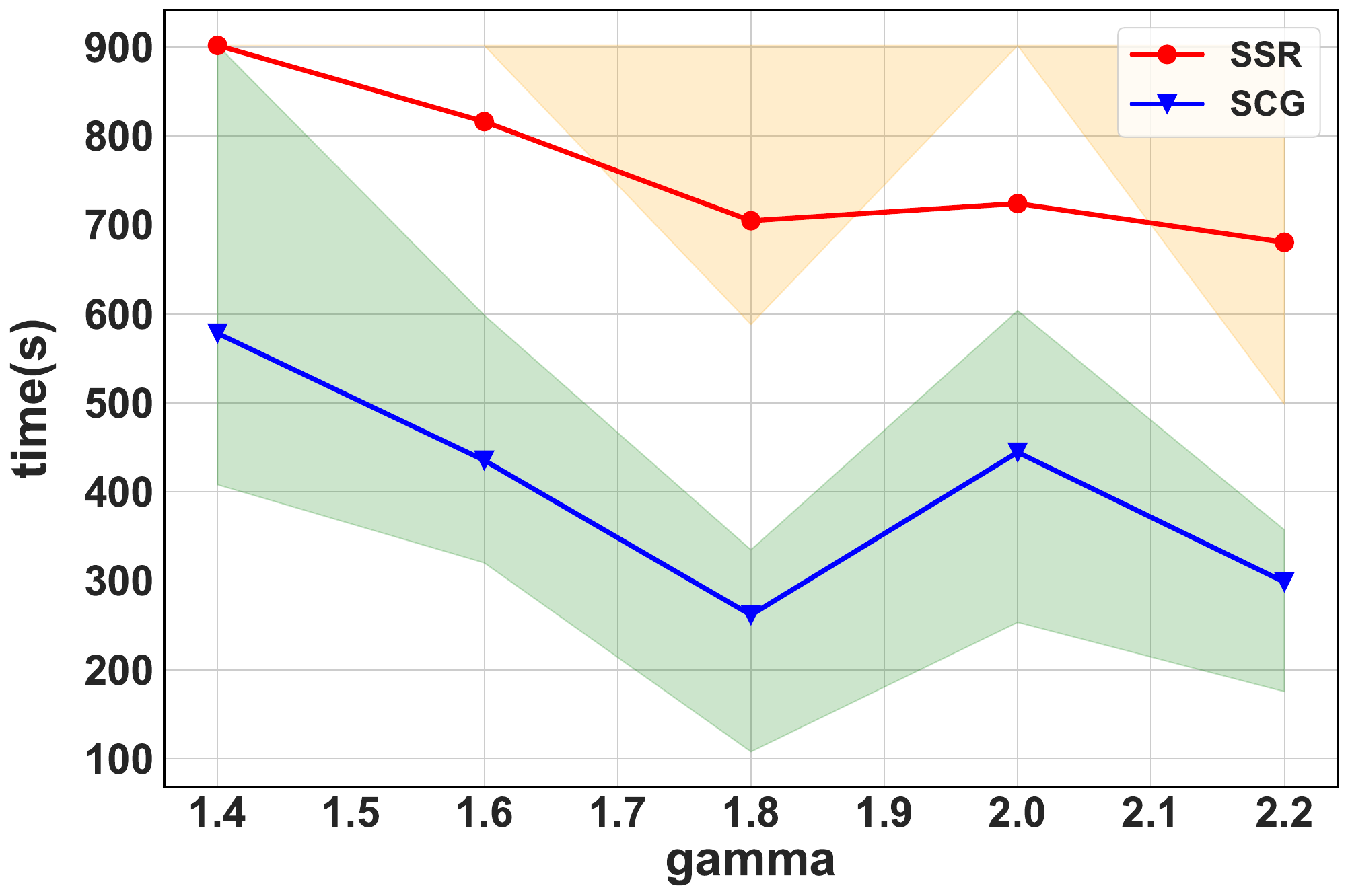}
\end{subfigure}
\begin{subfigure}
\centering
\includegraphics[width=0.32\linewidth]{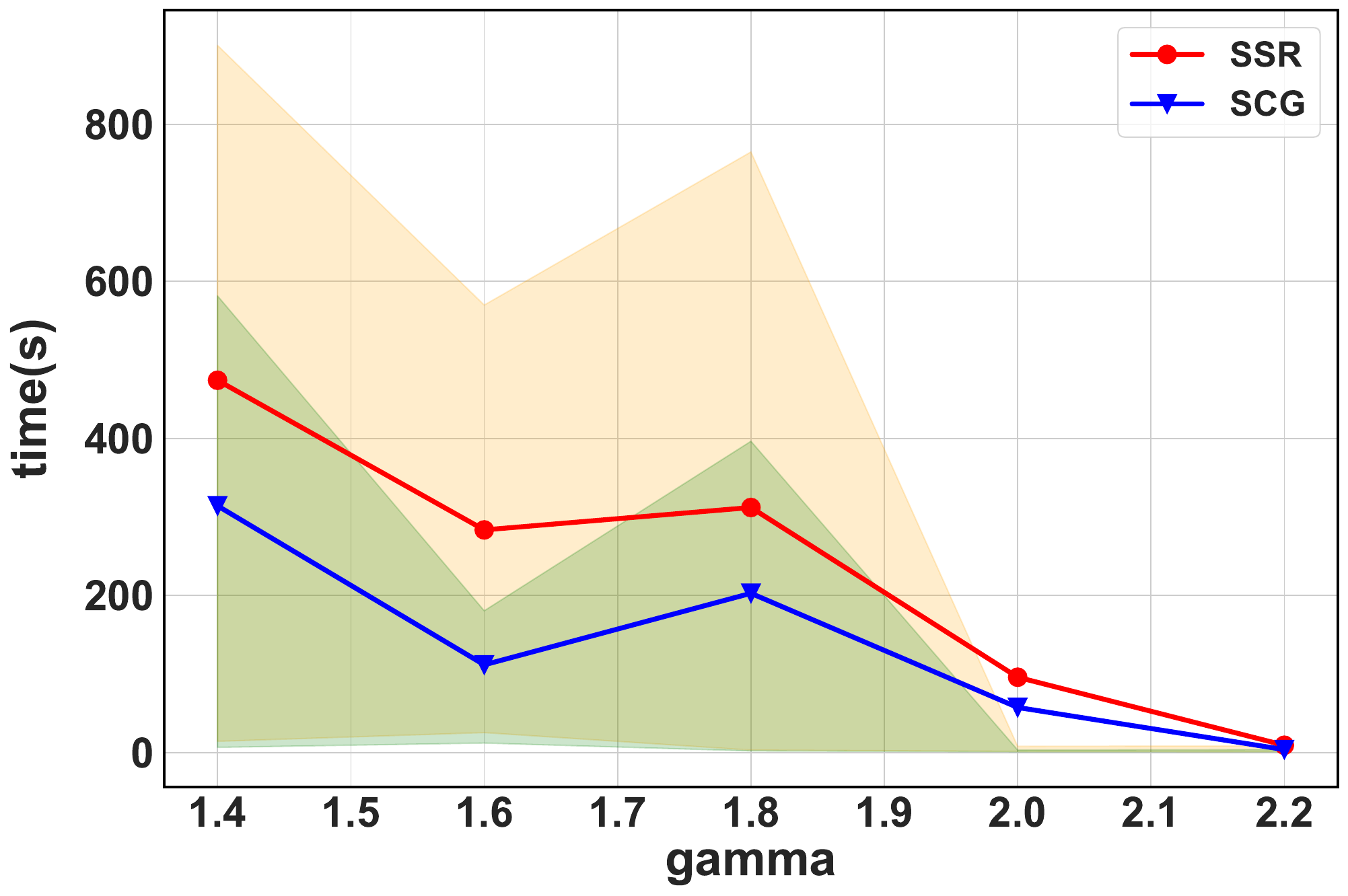}
\end{subfigure}
\begin{subfigure}
\centering
\includegraphics[width=0.32\linewidth]{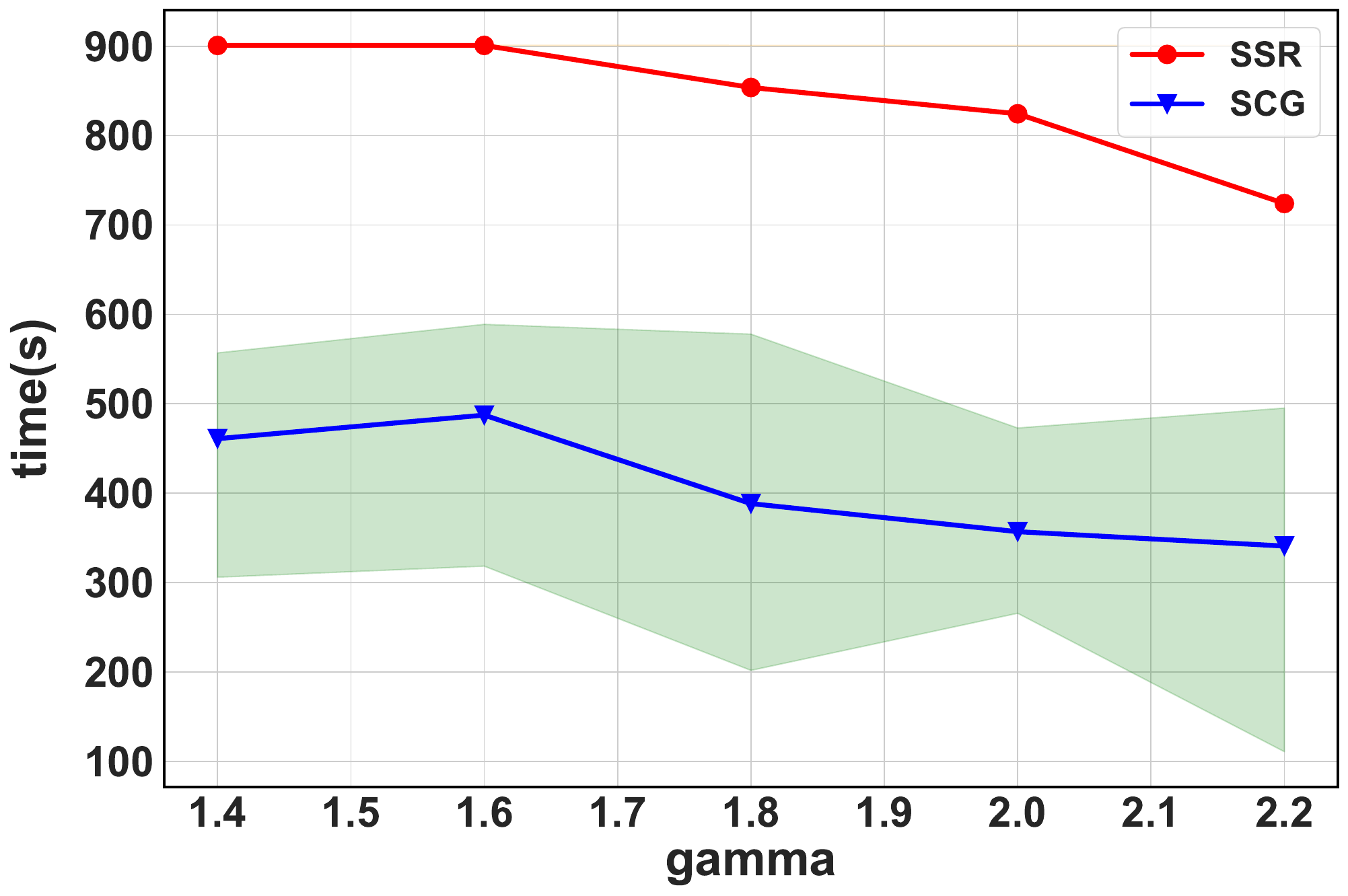}
\end{subfigure}
\begin{subfigure}
\centering
\includegraphics[width=0.32\linewidth]{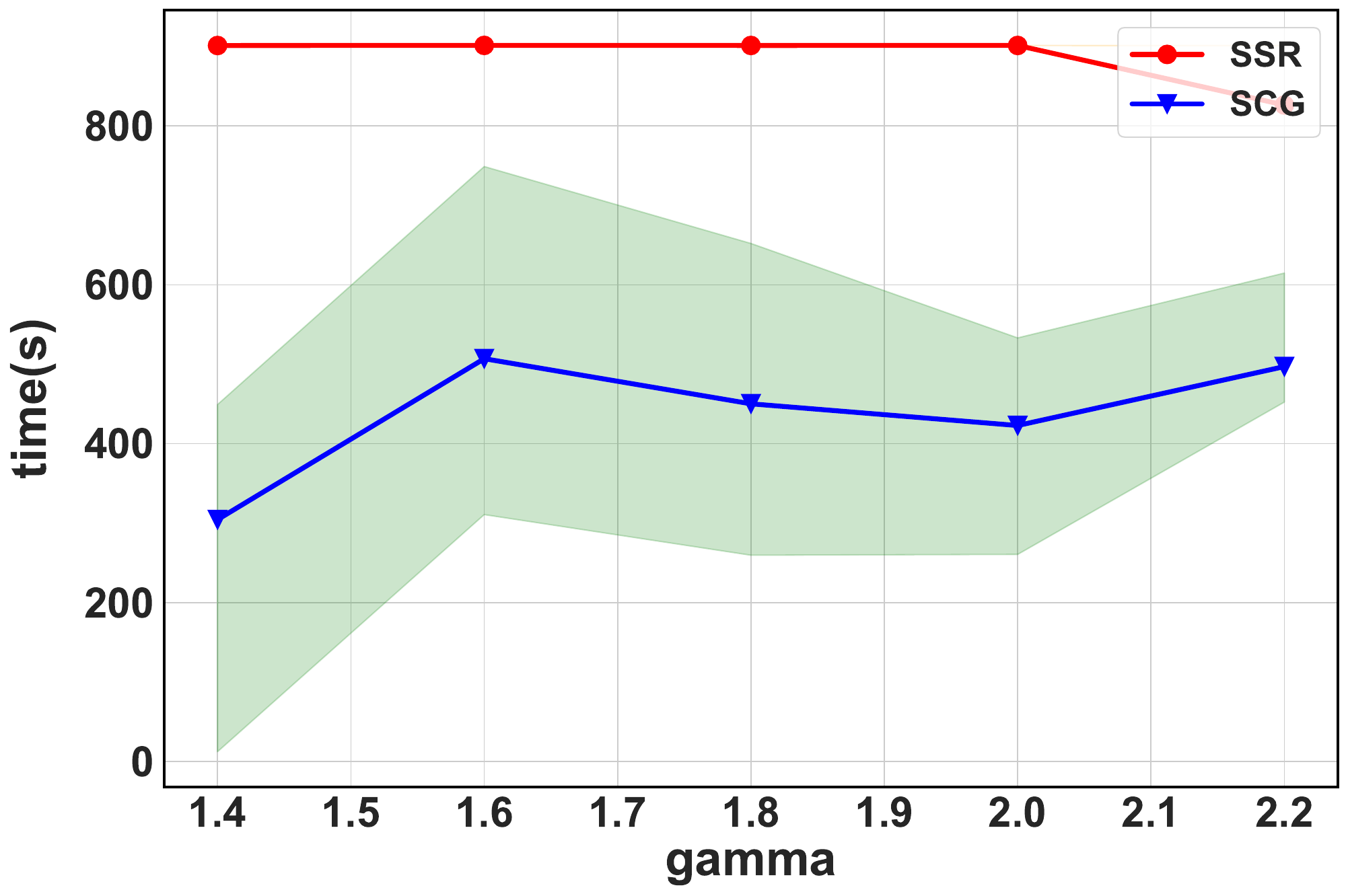}
\end{subfigure}
    \caption{\textbf{Total running time for synthetic datasets}. The parameter configurations are specified in Section \ref{sec:description-of-experimental-data}. The first and second row corresponds to $\text{SNR} = 3.5$ and $\text{SNR} = 0.5$, respectively. The three columns, from left to right, are set as $d = 100$, $d = 300$, $d = 500$, respectively. Each dot represents the average value over $10$ independent datasets. The shaded area represents the inter-quartile range (25th to 75th percentiles).}
\label{fig:synthetic-datasets}
\vskip -0.2in
\end{figure*}

\subsection{Discussions on Numerical Results}\label{sec:numerical-discussion}

\subsubsection{Factors That Influence SCG}

We start with figuring out which factors may greatly influence the effectiveness of selecting screening cuts, i.e., (i) \emph{number of screening cuts selected} and (ii) \emph{potential screening ability with respect to a selected screening cut} by SCG. Due to the space limit, the cut characteristic tables are provided in Appendix \ref{app:cuts-table}. Numerical results reported in Table~\ref{tbl:cuts_100_0.5}, \ref{tbl:cuts_100_3.5}, \ref{tbl:cuts_300_0.5}, \ref{tbl:cuts_300_3.5}, \ref{tbl:cuts_500_0.5}, \ref{tbl:cuts_500_3.5}, \ref{tbl:UJIndoor}, \ref{tbl:CNAE} show that:

\begin{enumerate}
    \item $\gamma$ value has a primary impact on the effectiveness of selecting screening cuts. In particular, as $\gamma$ decreases, the number of generated inclusive and exclusive cuts reduces for both SCG and SSR (see Table~\ref{tbl:cuts_100_0.5}, \ref{tbl:cuts_100_3.5}, \ref{tbl:cuts_300_0.5}, \ref{tbl:cuts_300_3.5}, \ref{tbl:cuts_500_0.5}, \ref{tbl:cuts_500_3.5}), which validates the theoretical result presented in Theorem~\ref{thm:generalized-cuts-generation-rule}. Additionally, we would like to point out that, in low $\gamma$ regime, \textit{with sufficient samples, SCG can still generate strong screening cuts, while SSR cannot.} See Table~\ref{tbl:UJIndoor}, \ref{tbl:CNAE} for details.
    
    \item Other factors, such as lower SNR and higher dimension $d$, have limited impacts on the effectiveness of selecting inclusive/exclusive cuts, compared with $\gamma$ value. As SNR decreases (see Table~\ref{tbl:cuts_100_0.5}, \ref{tbl:cuts_100_3.5}, and Table~\ref{tbl:cuts_300_0.5}, \ref{tbl:cuts_300_3.5}, and Table~\ref{tbl:cuts_500_0.5}, \ref{tbl:cuts_500_3.5}) or dimension increases (see Table~\ref{tbl:cuts_100_0.5}, \ref{tbl:cuts_300_0.5}, \ref{tbl:cuts_500_0.5}, and Table \ref{tbl:cuts_100_3.5}, \ref{tbl:cuts_300_3.5}, \ref{tbl:cuts_500_3.5}), the number of selected inclusive or exclusive cuts reduces for both SSR and SCG, and the length (number of binaries involved) of selected screening cuts increases for SCG. 
\end{enumerate}

In summary, SCG ensures stronger effectiveness in selecting inclusive or exclusive cuts than SSR, particularly in hard real datasets or synthetic datasets with challenging parameter configurations, as described above.

\subsubsection{Running Time Improved By SCG}

We then compare the total running time between the proposed SCG algorithm and the existing SSR. \\

\noindent \textbf{Synthetic dataset.} Figure \ref{fig:synthetic-datasets} shows that, as the dimension increases, the cuts selected by SCG lead to better numerical performances on solving \eqref{eq:ridge-regression} than the cuts screened by SSR under different choices of SNR values. 

In particular, the relative time gaps between SCG and SSR (defined as $\texttt{RTgap} := \frac{t_{\texttt{SSR}} - t_{\texttt{SCG}}}{t_{\texttt{SSR}}}$) keep increasing as dimension increases, i.e., \texttt{RTgap} is about $10 \% \sim 20 \%$ in $d = 100$, $20 \% \sim 30 \%$ in $d = 300$, and roughly $40 \%$ in $d = 500$. \\

\noindent \textbf{Real-world dataset.} Numerical results for CNAE (first two plots in Figure \ref{fig:real-datasets}) show that the running time of SSR and SCG heavily depends on the choice of $\gamma$. Specifically, as $\gamma$ decreases, the running time enjoys a ``\emph{three phase}'' trend, which can be explained as follows: 
\begin{enumerate}
    \item \textbf{Strong SCG v.s. Strong SSR.}  When $\gamma$ starts with a relatively large value, SSR could screen a significant part of binary variables, leading to comparable numerical performances in contrast with SCG.
    
    \item \textbf{Strong SCG v.s. Weak SSR.} When $\gamma$ decreases, we enter a parameter regime of $\gamma$ where SSR's screening ability is significantly weakened, while SCG maintains a good screening ability by generating effective inclusive and exclusive cuts in the pre-processing step of BnB. 

    \item \textbf{Weak SCG v.s. Weak SSR.} When $\gamma$ decreases to sufficiently small values, both methods' capabilities to generate minimal screening cuts become limited. Thus, BnB reduces to solve the original problem without any screening cuts generated in the pre-processing step.
\end{enumerate}

Numerical results for UJIndoorLoc (last two plots in Figure~\ref{fig:real-datasets}) show that the gaps between running times of SCG and SSR are more significant but unstable, compared with gaps for CNAE. Such unstable numerical performances may stem from the ill condition number of the raw UJIndoorLoc dataset. \\

\begin{figure}[t]
\begin{subfigure}
        \centering
    \includegraphics[width=0.235\linewidth]{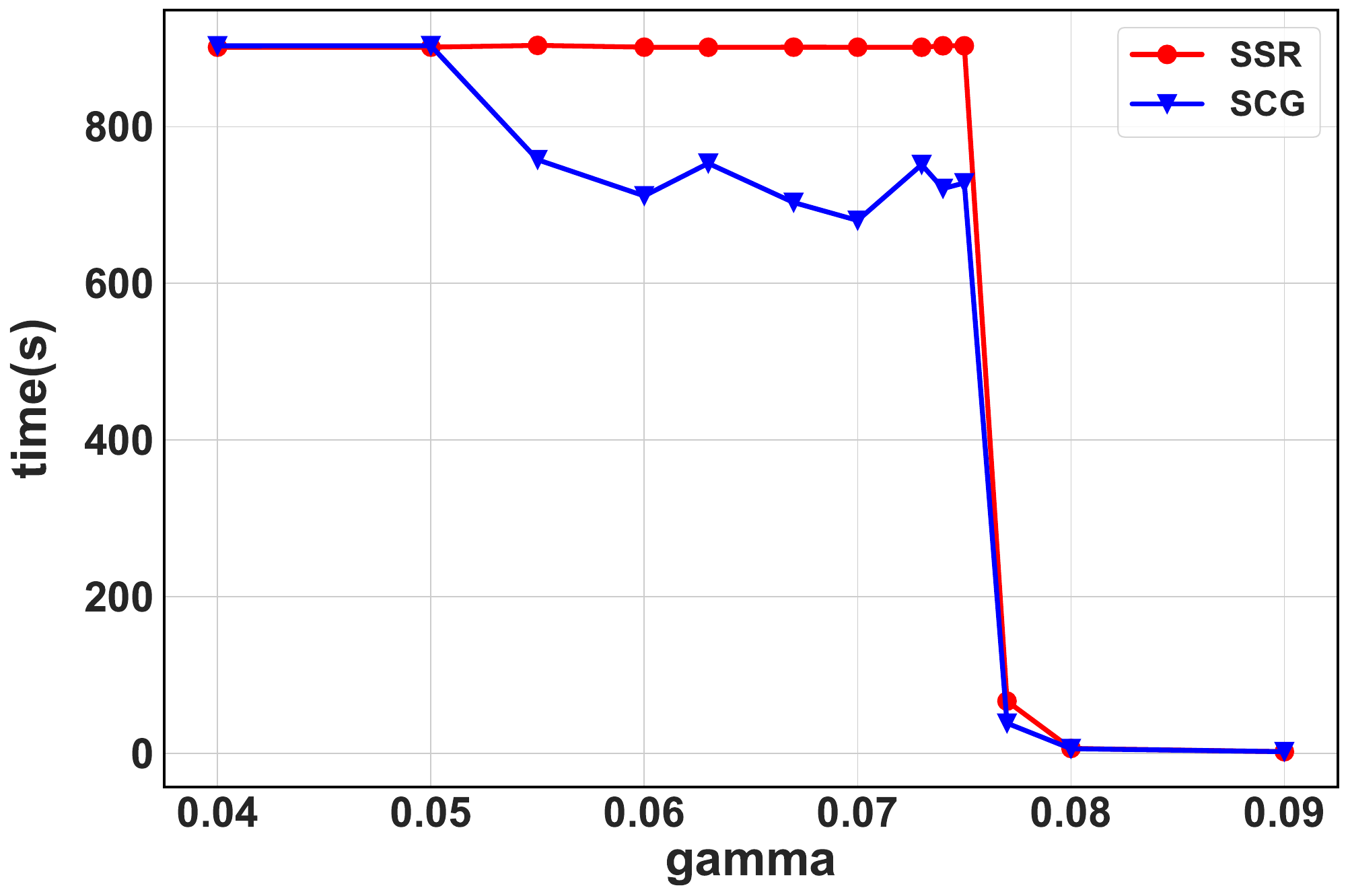}
    \end{subfigure}
    \hfill
    \begin{subfigure}
        \centering
    \includegraphics[width=0.235\linewidth]{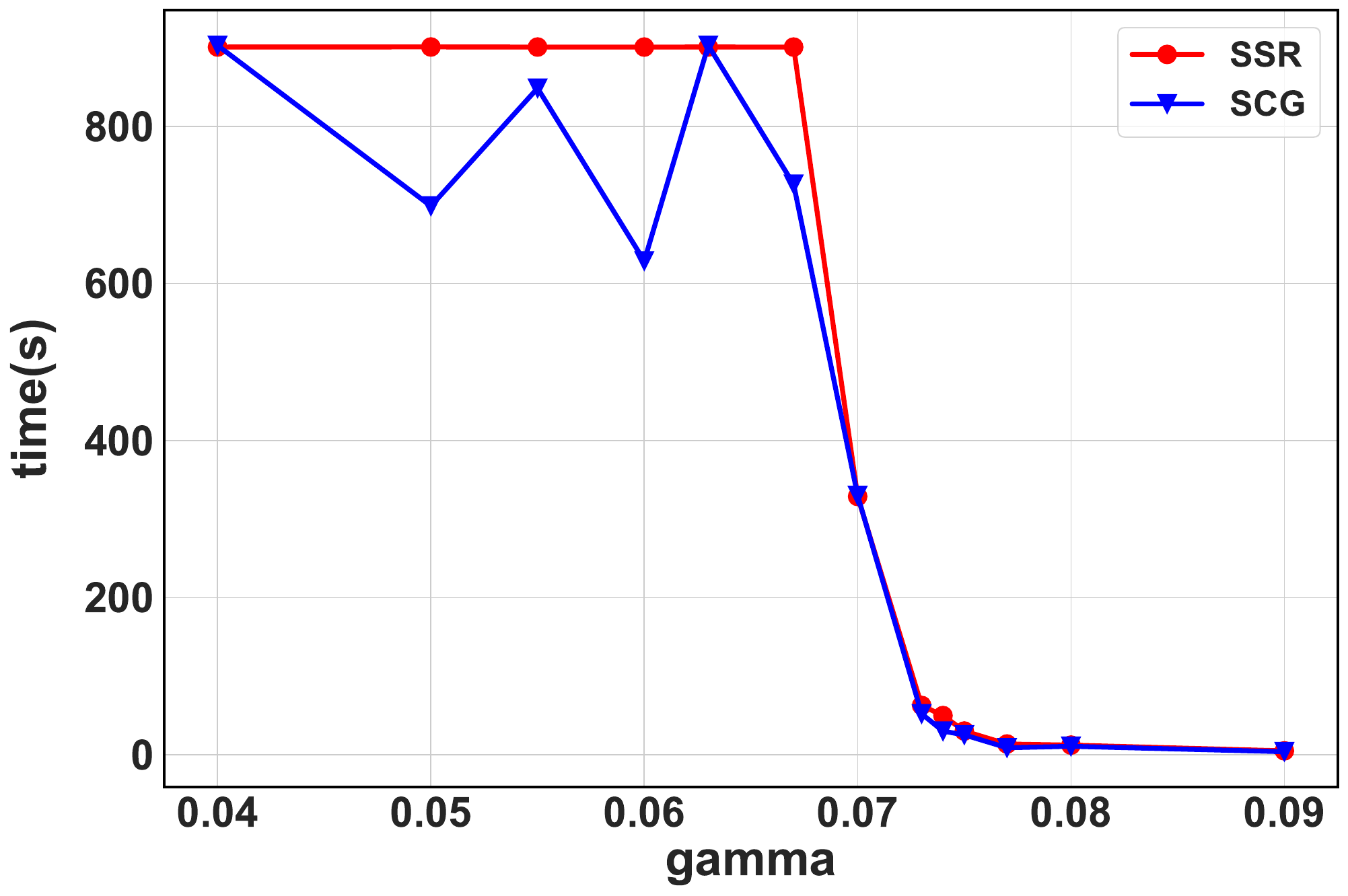}
    \end{subfigure}
    \hfill
    \begin{subfigure}
        \centering
    \includegraphics[width=0.235\linewidth]{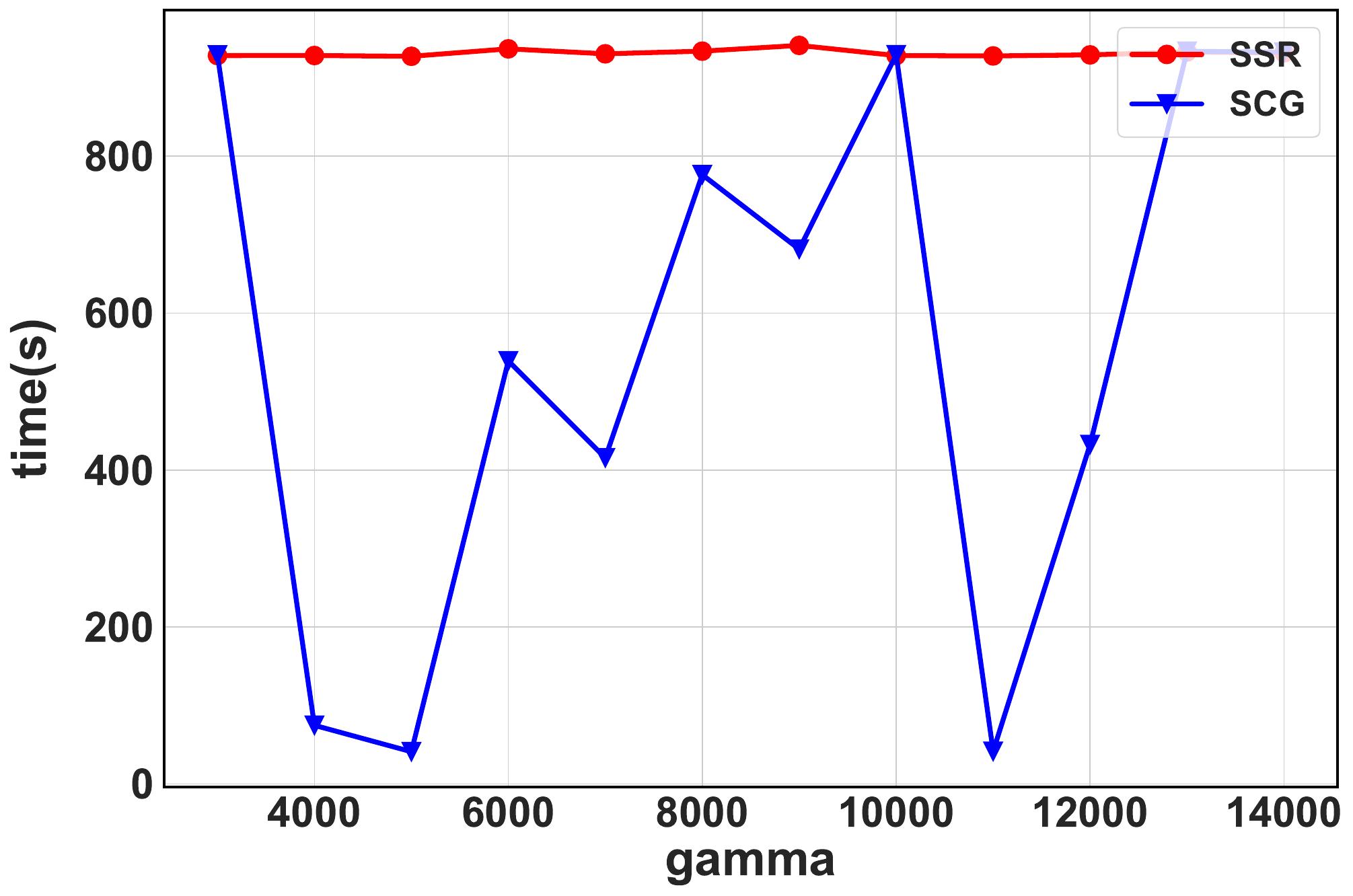}
    \end{subfigure}
    \hfill
    \begin{subfigure}
        \centering
        \includegraphics[width = 0.235\linewidth]{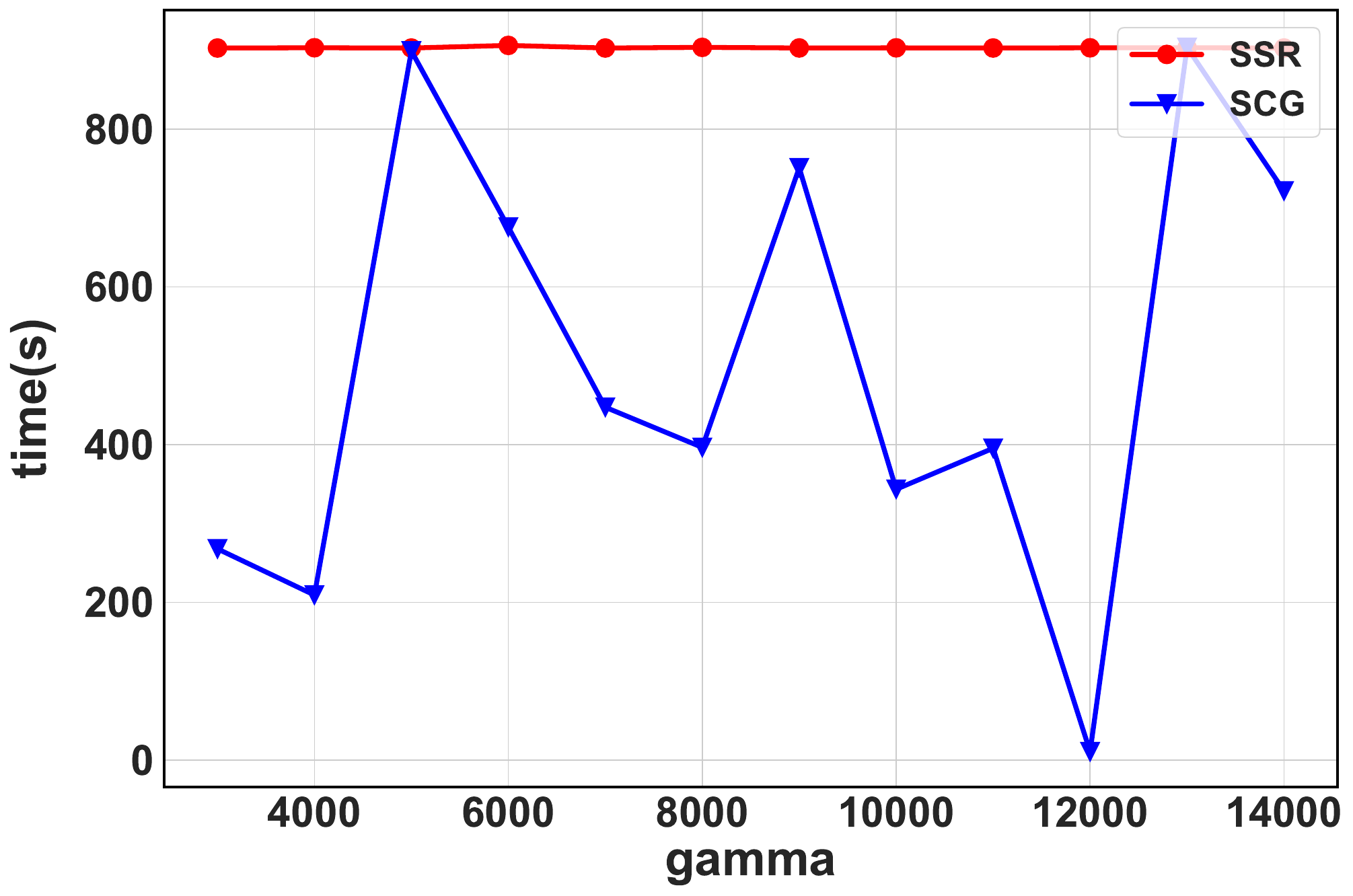}
    \end{subfigure}
\caption{\textbf{Total running time for real datasets}. The parameter configurations are specified in Section \ref{sec:description-of-experimental-data}. The four plots, arranged from left to right, represent the total running time under different $\gamma$ values with: whole CNAE dataset, reduced CNAE dataset with 200 samples, whole UJIndoorLoc dataset, reduced UJIndoorLoc dataset with 300 samples.}
\label{fig:real-datasets}
\end{figure}

\section{Conclusion \& Future Directions} \label{sec:conclusion}

In conclusion, this paper introduces a novel cut-generating method, Screening Cut Generation (SCG), to improve and enhance existing variable screening approaches used in the pre-processing step of BnB and its variants. We show that convex relaxation solutions can be more effectively leveraged to generate screening cuts that eliminate both single non-optimal binary variables and specific $\{\pm 1\}$ combinations of multiple binaries that are not part of any optimal solutions. Furthermore, we establish necessary and sufficient conditions for certifying minimal (dominating) screening cuts. Combined with the cut-selection criteria based on potential screening ability, these theoretical insights form the foundation for our screening cut generation algorithm. Numerical experiments in Section~\ref{sec:numerical} further validate the efficiency and effectiveness of the proposed method. We close with some potential SCG extensions and research questions for future investigation. The first is to study the closure of minimal screening cuts generated by every feasible point in the $k$-cardinality binary knapsack polytope \eqref{eq:knapsack-set}. Second, it is natural to extend existing screening approaches to constrained sparse ridge regression. Lastly, motivated by some preliminary numerical results, whether one can generate screening cuts that are valid under a low $\gamma$ regime case based on ``stabilized'' supports for sufficiently small $\gamma_i$ as discussed after Theorem~\ref{thm:generalized-cuts-generation-rule} remains an open and compelling question.

\subsection*{Acknowledgments}
Haozhe Tan and Guanyi Wang were supported by the Ministry of Education, Singapore, under the Academic Research Fund (AcRF) Tier-1 grant (A-8000607-00-00) 22-5539-A0001. 

\bibliographystyle{abbrvnat}
\bibliography{reference-template}

\newpage
\appendix
\onecolumn

\section{Proofs in Section~\ref{sec:main-results}}
\subsection{Proof of Theorem~\ref{thm:generalized-cuts-generation-rule}} \label{app:prop-generalized-cuts-generation-rule}

\optcutgen*

\begin{proof}
    Adding constraints $z_i = 1$ for all $i \in S$ and $z_j = 0$ for all $j \in N$ to sparse ridge regression~\eqref{eq:original-formulation} gives 
    \begin{align*}
        \begin{array}{rll}
            v_{S,N}:=\underset{\bm{\beta},\bm{z}}{\operatorname{min}}& \mathcal{L}(\bm{\beta}) + \gamma \|\bm{\beta}\|_2^2\\
            \text{s.t.} & (1-z_i)\beta_i = 0~~\forall i \in [\![d]\!] \\
            &\bm{1}^{\top} \bm{z} \leq k, ~ \bm{z} \in \{0, 1\}^d \\
             & z_i = 1 ~ \forall ~ i \in S, ~~ z_i = 0 ~ \forall ~ i \in N
        \end{array}
    \end{align*}
    and its naive conic relaxation becomes:
    \begin{align*}
        \begin{array}{rll}
            \tilde{v}_{S, N} := \underset{\bm{p}}{\operatorname{max}}~ \underset{\bm{\beta},\bm{z}}{\operatorname{min}}& \mathcal{L}(\bm{\beta}) + \gamma \sum_{j=1}^d (p_j\beta_j - \frac{p_j^2}{4}z_j) \\
            \text{s.t.} & \bm{z} \in \mathcal{Z}_{S, N} := \left\{ \bm{z} \in [0, 1]^d ~|~ \bm{1}^{\top} \bm{z} \leq k, ~ z_i = 1 ~ \forall ~ i \in S, ~~ z_i = 0 ~ \forall ~ i \in N \right\}
        \end{array}
    \end{align*}
    A lower bound of $\tilde{v}_{S, N}$ can be obtained by plugging in $\hat{\bm{p}}$ (obtained from \eqref{eq:fenchel-relaxation}) to the outer maximization: 
        \begin{align} \label{eq:restrcited_lower_bound}
        \begin{array}{rll}
            v_{\tt res}:=\underset{\bm{\beta}, ~ \bm{z} \in \mathcal{Z}_{S, N}}{\operatorname{min}}& \mathcal{L}(\bm{\beta}) + \gamma \sum_{j=1}^d (\hat{p}_j\beta_j - \frac{\hat{p}_j^2}{4}z_j)
        \end{array} ~~. 
    \end{align}
    Note the objective function of \eqref{eq:restrcited_lower_bound} can be separated into two parts:
    \begin{align*}
        \mathcal{L}(\bm{\beta}) + \gamma \sum_{j=1}^d (\hat{p}_j\beta_j - \frac{\hat{p}_j^2}{4}z_j) = \bigg( \underbrace{\mathcal{L}(\bm{\beta}) + \gamma \sum_{j = 1}^d \hat{p}_j\beta_j}_{=: f_1(\bm{\beta}) }  \bigg) - \bigg( \underbrace{ \gamma \sum_{j = 1}^d \frac{\hat{p}_j^2}{4}z_j }_{=: f_2(\bm{z})}\bigg)  ~~, 
    \end{align*}
    where, given $\hat{\bm{p}}$, the first part $f_1(\bm{\beta})$ only depends on $\bm{\beta}$ and the second part $f_2(\bm{z})$ only depends on $\bm{z}$. Thus, minimizing \eqref{eq:restrcited_lower_bound} is equivalent to minimizing the first part $f_1(\bm{\beta})$ and maximizing the second part $f_2(\bm{z})$ at the same time, i.e., 
    \begin{align*}
        v_{\tt res} ~ = ~ \min_{\bm{\beta}} ~ f_1(\bm{\beta}) ~ - ~ \max_{\bm{z} \in \mathcal{Z}_{S, N}} ~ f_2(\bm{z}) 
    \end{align*}
    To minimize the first part $f_1(\bm{\beta})$, it is easy to verify that $\hat{\bm{\beta}}$ obtained from \eqref{eq:conic-relaxation} is also an optimal solution. To maximize linear objective function $f_2(\bm{z})$, we need to choose at most $k$ largest values of $\frac{\gamma}{4}\hat{p}_j^2 = \frac{\gamma}{4}\hat{w}_j$ while maintaining the feasibility requirement, i.e., $\bm{z} \in \mathcal{Z}_{S, N}$. In particular, consider the following two cases: 
    
    \begin{enumerate}
        \item Suppose $|N| > d-k$, then $\bm{1}^T \bm{z} \leq k$ is automatically satisfied.  By letting $z_i = 1$ for all $i \in [\![d]\!] \setminus N$ that includes $S$ as a subset, the resulting tuple $(S, N, [\![d]\!] \setminus (S \cup N))$ is a SCG-tuple with $S \cup C = [\![d]\!] \setminus N$ forms the support of optimal $\bm{z}$.

        \item Suppose $|N| \leq d - k$, then the remaining index set $R := [\![d]\!] \setminus (S \cup N) $ includes at least $k - |S|$ elements. Thus, to maximize $f_2(\bm{z})$, the  optimal $\bm{z}$ selects top-$(k - |S|)$ largest components in $\{\frac{\gamma}{4}\hat{w}_j\}_{j \in R}$. Still, the resulting $(S,N,C)$ is a SCG-tuple with respect to $\hat{\bm{w}}$, and set $S \cup C$ forms the support of optimal $\bm{z}$.
    \end{enumerate}

    Combining the above two cases ensures that the optimal $\bm{z}$ always takes $z_i = 1$ for $i \in S \cup C$. Then, we have
    \begin{align*}
        v_{S, N} \geq \tilde{v}_{S, N} \geq v_{\tt res} = & ~ \mathcal{L}(\hat{\bm{\beta}}) + \gamma \sum_{j=1}^d (\hat{p}_j \hat{\beta}_j) - \gamma \sum_{j \in (S \cup C)} \frac{1}{4}\hat{w}_j \\
        = & ~ \bigg[ \underbrace{\mathcal{L}(\hat{\bm{\beta}}) + \gamma \sum_{j=1}^d (\hat{p}_j \hat{\beta}_j) - \gamma \sum_{j = 1}^k \frac{1}{4}\hat{w}_{[j]} }_{= v_{\tt conic}} \bigg] + \frac{\gamma}{4} \bigg[  \sum_{i=1}^k \hat{w}_{[i]} - \sum_{i \in S} \hat{w}_i - \sum_{j \in N} \hat{w}_j \bigg] 
    \end{align*}
    If inequality~\eqref{ineq:cuts-condition} holds, we have $\frac{\gamma}{4} [  \sum_{i=1}^k \hat{w}_{[i]} - \sum_{i \in S} \hat{w}_i - \sum_{j \in N} \hat{w}_j ] > v_{\tt ub} - v_{\tt conic}$, which implies $v_{S,N} \geq v_{\tt res} > v_{\tt ub} \geq v^*$. That is to say, any feasible solution of \eqref{eq:original-formulation} with $z_i = 1$ for $i \in S$ and $z_j = 0$ for $j \in N$ cannot be optimal, and we can add screening cut $\sum_{i \in S}z_i + \sum_{j \in Z}(1-z_j) \leq |S| + |Z| - 1$ to remove all points that satisfy $\bm{z} \in \mathcal{Z}_{S, N}.$
    
    Notice that if there exists some $i \in C$ with $\frac{\gamma}{4}\hat{w}_i = 0$, any $z_i \in [0,1]$ leads to an optimal solution for \eqref{eq:restrcited_lower_bound}. Without loss of generality, we always choose $z_i = 1$ for our analysis.
\end{proof}

\subsection{Proof of Remark~\ref{rmk:AA_rule}} \label{app:AA_rule}

\begin{proof}
    let $S = \{i\}$ for some $i \in [\![d]\!]$ and $N = \emptyset$. If $i \in \texttt{Top}_k(\hat{\bm{w}})$, we have $S \cup C = \texttt{Top}_k(\hat{\bm{w}})$. The inequality (\ref{ineq:cuts-condition}) is not satisfied with $0 > \frac{4}{\gamma} \texttt{gap}$ and no cut can be generated. If $i \notin \texttt{Top}_k(\hat{\bm{w}})$, and the inequality (\ref{ineq:cuts-condition}) is satisfied, which is equivalent to $\hat{w}_{[k]} - \hat{w}_i > \frac{4}{\gamma} \tt gap$, then $(\{i\},\emptyset,\texttt{Top}_{k-1}(\hat{\bm{w}}))$ is a valid tuple. The resulting cut is $z_i \leq 0$. 
    
    Similarly, let $S = \emptyset$ and $Z = \{j\}$ for some $j \in [\![d]\!]$. If $j \notin \texttt{Top}_k(\hat{\bm{w}})$, we have $S \cup C = \texttt{Top}_k(\hat{\bm{w}})$. The inequality (\ref{ineq:cuts-condition}) is not satisfied with $0 > \frac{4}{\gamma} \texttt{gap}$ and no cut can be generated. If $j \in \texttt{Top}_k(\hat{\bm{w}})$, and the inequality (\ref{ineq:cuts-condition}) is satisfied, which is equivalent to $\hat{w}_{j} - \hat{w}_{[k+1]} > \frac{4}{\gamma} \tt gap$, then $(\emptyset,\{i\},\texttt{Top}_{k+1}(\hat{\bm{w}}) \setminus \{j\})$ is a valid tuple. The resulting cut is $z_j \geq 1$. 
\end{proof}

\subsection{Proof of Corollary~\ref{coro:valid-tuple}} \label{app:valid-tuple} 

Here, we propose a corollary of Theorem~\ref{thm:generalized-cuts-generation-rule}, which will be used in the proof of Proposition~\ref{prop:minimal-tuple}. 

\begin{restatable}{corollary}{corovalid}
\label{coro:valid-tuple}  
Suppose components in $\hat{\bm{w}} = \hat{\bm{p}} \circ \hat{\bm{p}}$ are distinct. Given a SCG-tuple $(S, N, C)$, we say it is valid if the following three conditions hold simultaneously: 
\begin{enumerate}
    \item[(1)] $S,N,C \subseteq [\![d]\!]$ are pair-wisely disjoint with $|C| = \min \{k - |S|, d - |N| - |S| \}$;
    \item[(2)] If $C \neq \emptyset$, then for all $i$ such that $\hat{w}_i \geq \underline{\hat{w}_C}$, we have $i \in (S\cup N\cup C)$;
    \item[(3)] $\sum_{i=1}^k \hat{w}_{[i]} - \sum_{i \in S} \hat{w}_i - \sum_{j \in C} \hat{w}_j > \frac{4}{\gamma} \gap$. 
\end{enumerate}
\end{restatable}

\begin{proof}
    Denote $R$ as the set of remaining indices, i.e. $R = [\![d]\!] \setminus (S \cup N)$.
    
    $\Rightarrow$: Define $C:=\{i \in R~|~ \hat{w}_i \geq (\hat{w}_R)_{[c]}\}$, where $c:=\operatorname{min} \{k - |S|,d - |N| - |S|\}$. If $|S| = k$, or $d = |N| + |S|$, then $C = \emptyset$ and we are done. For other cases, $c >0$ and $C \neq \emptyset$ since $c \leq d - |N| - |S| = |R|$. Clearly $|C| = c$ and $\underline{\hat{w}_C} = (\hat{w}_R)_{[c]}$. We only need to show point (ii). Prove by contradiction. Suppose there exists some $j \notin (S\cup N\cup C)$ such that $\hat{w}_j \geq \underline{\hat{w}_C}$. First, we have $j \in R$. We also have $j \in C$ since $\hat{w}_j \geq \underline{\hat{w}_C} = (\hat{w}_R)_{[c]}$. Contradiction.

    $\Leftarrow$: If $|S| = k$, or $d = |N| + |S|$, then $C = \emptyset$ and we are done. For other cases, we have $|C| >0$. Define $\mathcal{D}:=\{i \in R~|~ \hat{w}_i \geq (\hat{w}_R)_{[|C|]}\}$. This is valid since $|C| \leq d - |N| - |S| = |R|$. Now we know $|\mathcal{D}| = |C|$. It is sufficient to prove that pick any $i' \in C$, $i' \in D$. Clearly $i' \in R$. Suppose $i' \notin \mathcal{D}$, we have $\hat{w}_{i'} < (\hat{w}_R)_{[|C|]}$. Since we know $|\mathcal{D}| = |C|$, the existence of $i'$ implies that there exists some $j' \in \mathcal{D}$ but $j' \notin C$. Thus we have $\hat{w}_{j'} \geq ({\hat{w}_R})_{[|C|]} > \hat{w}_{i'} \geq \underline{\hat{w}_C}$, and $j' \notin (S\cup N\cup C)$. Contradiction with the point (ii). Thus we have $C = \mathcal{D}$ and we are done.
\end{proof}

\subsection{Proof of Proposition~\ref{prop:cut-domination}} \label{app:cut-domination}

\cutdomination*

\begin{proof}
The screening cuts generated by $(S_1, N_1, C_1), (S_2, N_2, C_2)$ are 
\begin{align*}
    \sum_{i \in S_1} z_i + \sum_{j \in N_1} (1 - z_j) \leq |S_1| + |N_1| - 1  \quad \text{and} \quad \sum_{i \in S_2} z_i + \sum_{j \in N_2} (1 - z_j) \leq |S_2| + |N_2| - 1 ~~, 
\end{align*}
respectively. Consider any binary vector $\tilde{\bm{z}} \in \{0,1\}^d$ that is feasible for the screening cuts generated by $(S_1, N_1, C_1)$, to show domination, it is sufficient to have
\begin{align*}
    & ~ \sum_{i \in S_2} \tilde{z}_i + \sum_{j \in N_2} (1 - \tilde{z}_j) \leq |S_2| + |N_2| - 1 \\
    \Leftrightarrow ~~&~ \sum_{i \in S_2} \tilde{z}_i \leq |S_2| - 1 + \sum_{j \in N_2} \tilde{z}_j \\
    \Leftrightarrow ~~&~ \sum_{i \in S_1} \tilde{z}_i + \sum_{i \in S_2 \setminus S_1} \tilde{z}_i \leq (|S_1| - 1) + (|S_2| - |S_1|) + \sum_{j \in N_1} \tilde{z}_j + \sum_{j \in N_2 \setminus N_1} \tilde{z}_j \\
    \Leftrightarrow ~~&~ \sum_{i \in S_1} \tilde{z}_i + \sum_{i \in S_2 \setminus S_1} \tilde{z}_i \leq \left(|S_1| - 1 + \sum_{j \in N_1} \tilde{z}_j \right) + (|S_2| - |S_1|) + \sum_{j \in N_2 \setminus N_1} \tilde{z}_j ~~, 
\end{align*} 
where the last inequality holds due to our assumption on $\tilde{z}$, $\sum_{i \in S_2 \setminus S_1} \tilde{z}_i \leq |S_2| - |S_1|$, and $\sum_{j \in N_2 \setminus N_1} \tilde{z}_j \geq 0$. Therefore, the domination holds. 
\end{proof}

\subsection{Proof of Proposition \ref{prop:minimal-tuple}}
\label{app:minimal-tuple}

\minimal*
\begin{proof}
Before presenting the proof of Proposition \ref{prop:minimal-tuple}, recall the following three conditions of a \emph{valid SCG-tuple} $(S,N,C)$ with respect to $\hat{\bm{w}}$:
\begin{enumerate}
    \item [(1)] $S,N,C \subseteq [\![d]\!]$ are pair-wisely disjoint with $|C| = \min \{k - |S|, d - |N| - |S| \}$; 
    \item [(2)] If $C \neq \emptyset$, for all $i$ with $\hat{w}_i \geq \underline{\hat{w}_C}$, we have $i \in (S\cup N\cup C)$; 
    \item [(3)]$\sum_{i=1}^k \hat{w}_{[i]} - \sum_{i \in S} \hat{w}_i - \sum_{j \in C} \hat{w}_j > \frac{4}{\gamma} \gap $. 
\end{enumerate}

In the following proof, given any non-empty set $C \subseteq [\![d]\!]$, we denote $\mathcal{G}(C):=\{i \in [\![d]\!]~|~\hat{w}_i \geq \underline{\hat{w}_C}\}$, and thus, the second point (ii) above can be written as: if $C \neq \emptyset$, $\mathcal{G}(C) \subseteq S \cup N \cup C$. \\

\noindent \textbf{Proof of Case 1.}  

Let $(S,N,C)$ is a valid SCG-tuple with $|N| > d - k$. Based on Definition~\ref{def:SCG-tuple}, we have $S \cup N \cup C = [\![d]\!]$ and $|S \cup C| < k$. 

$\Leftarrow$:
If $S = \emptyset$, we have $C = [\![d]\!] \setminus N$ and $|C| = d - |N| < k$. Based on Definition~\ref{def:minimal-tuple}, it is sufficient to show that $(\emptyset,N \setminus \{i\}, C)$ cannot form a valid tuple for any $i \in N$. To show invalidity, for any $i \in N$, we claim that the tuple $(\emptyset,N \setminus \{i\},C)$ is invalid since it violates point (i): $|C| \neq \operatorname{min}\{k - |\emptyset|, d - |N \setminus \{i\}| - |\emptyset|\}$, where $\operatorname{min}\{k - |\emptyset|, d - |N \setminus \{i\}| - |\emptyset|\} = d - |N| + 1$. Therefore, $(\emptyset,N,C)$ is minimal. 

$\Rightarrow$:
If $(S,N,C)$ is minimal, by Definition~\ref{def:minimal-tuple}, we have either $S = \emptyset$ or $(S \setminus \{i\},N,C \cup \{i\})$ is not valid for any $i \in S$. Suppose $C$ is non-empty. For each $i \in S$, $(S \setminus \{i\},N,C \cup \{i\})$ is valid, since (i) $|C \cup \{i\}| = |C| + 1,~\operatorname{min}(k - |S \setminus \{i\}|, d-|N| - |S \setminus \{i\}|) = d - |N| - |S| +1$ and we know $|S| + |N| + |C| = d$; (ii) it is obvious since $S \cup N \cup C = [\![d]\!]$; (iii) since $(S,N,C)$ is a valid SCG-tuple, the inequality also holds for $(S \setminus \{i\},N,C \cup \{i\})$. Thus $S = \emptyset$. \\

\noindent \textbf{Proof of Case 2.} 

Assume $(S,N,C)$ is a valid SCG-tuple with $|N| \leq d - k$ and $C = \emptyset$. Based on Definition~\ref{def:SCG-tuple} for SCG-tuple, we have $|S| = k$. Let $i^*$ be the index with respect to the largest component in $\hat{\bm{w}}$, i.e., $\hat{w}_{i^*} = \hat{w}_{[1]}$. We have the following condition $\operatorname{max}_{i \in S} \hat{w}_i < \hat{w}_{[1]}$ is equivalent to say $i^* \notin S$.

$\Leftarrow$: 
If $N = \emptyset$ and $i^* \notin S$, it is sufficient to show that $(S \setminus \{i\},\emptyset,\{i\})$ is not a valid SCG-tuple for any $i \in S$. This is true because it violates the second point: (ii) for any $j$ such that $\hat{w}_j \geq \hat{w}_i$, we have $j \in (S \setminus \{i\} \cup \emptyset \cup \{i\}) = S$. However,  $\hat{w}_{i^*} \geq \hat{w}_i$ and $i^* \notin S$ by assumption. Contradiction.

$\Rightarrow$: 
If $(S,N,\emptyset)$ is minimal, we show two conditions listed in Case 2, respectively, by contradiction.  
\begin{itemize}
    \item Suppose $N \neq \emptyset$. Then $(S,N \setminus \{i\},\emptyset)$ is a valid SCG-tuple for any $i \in N$, since (i) $|\emptyset| = 0$ and $\operatorname{min} (k - |S|,d - |N \setminus \{i\}| - |S|) = k - |S| = 0$; (ii) no need to verify; (iii) since $(S,N,\emptyset)$ is a valid tuple, the inequality also holds for $(S,N \setminus \{i\},\emptyset)$. Contradiction with $(S,N,\emptyset)$ being minimal. Thus $N = \emptyset$. 
    \item Suppose $i^* \in S$. We claim that $(S \setminus \{i^*\},\emptyset,\{i^*\})$ is also a valid SCG-tuple, since (i) $|\{i^*\}| = 1$ and $\operatorname{min} (k - |S\setminus \{i^*\}|,d - |\emptyset| - |S \setminus \{i^*\}|) = k - |S \setminus \{i^*\}| = k - |S| + 1 = 1$; (ii) $C = \{i^*\}$, and it's obvious; (iii) since $(S,\emptyset,\emptyset)$ is a valid tuple, the inequality also holds for $(S \setminus \{i^*\},\emptyset ,\{i^*\})$. Contradiction with $(S,\emptyset,\emptyset)$ being minimal. Thus we have $i^* \notin S$.
\end{itemize}

\noindent \textbf{Proof of Case 3.} 

Assume $(S,N,C)$ is a valid SCG-tuple with $|N| \leq d - k$ and $C \neq \emptyset$. Based on validity, we have $|S \cup C| = k$. Besides, $S$ and $N$ cannot be empty at the same time, in order to generate a valid screening cut. Therefore, we consider the following three scenarios. \\

\noindent\textbf{1. The first scenario: $S = \emptyset$, $N \neq \emptyset$.} It is sufficient to focus on the $N$ conditions. 

$\Leftarrow$:
 Suppose $\underline{\hat{w}_N} \geq \underline{\hat{w}_C}$. Then for any $i \in N$, $(\emptyset,N \setminus \{i\},C)$ is not valid since it violates the second point (ii) $i \notin \left( (\emptyset \cup N \setminus \{i\}) \cup C \right)$ and $\hat{w}_i \geq \underline{\hat{w}_C}$.

$\Rightarrow$:
Suppose, for any $i \in N$, we have $(\emptyset,N \setminus \{i\},C)$ is not valid, then we want to show $\underline{\hat{w}_N} \geq \underline{\hat{w}_C}$. Prove by contradiction. Suppose there exists some $j \in N$ such that $\hat{w}_j < \underline{\hat{w}_C}$. Then $(\emptyset, N \setminus \{j\},C)$ is a valid tuple by verifying that (i) $|C| = \operatorname{min} \{k, d - |N \setminus \{j\}|\} = k$; (ii) We have $\mathcal{G}(C) \subseteq (N \cup C)$ based on the validness of SCG-tuple $(\emptyset,N,C)$. Based on the assumption on $j$, we have $j \notin \mathcal{G}(C)$, and thus $\mathcal{G}(C) \subseteq ((N \setminus \{j\}) \cup C)$; (iii) the inequality holds since SCG-tuple $(\emptyset,N,C)$ is a valid tuple by assumption. Then, we get a contradiction which implies $\underline{\hat{w}_N} \geq \underline{\hat{w}_C}$. \\

\noindent\textbf{2. The second scenario: $S \neq \emptyset$, $N = \emptyset$.} It is sufficient to focus on the $S$ conditions.

$\Leftarrow$:
If there exists some $j \notin (S \cup \emptyset \cup C)$ such that $\underline{\hat{w}_C} > \hat{w}_j > \operatorname{max}_{i \in S} \hat{w}_i$, we want to show that for any $ i' \in S$, $(S \setminus \{i'\},\emptyset, C \cup \{i'\})$ is not valid. This is true, since it violates the second point (ii) we have $\hat{w}_j > \operatorname{max}_{i \in S} \hat{w}_i \geq \hat{w}_{i'} \geq \underline{\hat{w}_{C \cup \{i'\}}}$ and $j \notin (S \cup C)$.

$\Rightarrow$:
If $(S \setminus \{i\},\emptyset,C \cup \{i\})$ is invalid for all $i \in S$, we show conditions listed in Case 3 using contradiction. 
\begin{itemize}
    \item Suppose there exists some $i' \in S$ such that $\hat{w}_{i'} \geq \underline{\hat{w}_C}$. Then $(S \setminus \{i'\}, \emptyset, C \cup \{i'\})$ is valid, by verifying that (i) $|C \cup \{i'\}| = |C| + 1$, and $\operatorname{min}(k - |S \setminus \{i'\}|, d - |S \setminus \{i'\}|) = k - |S| + 1$. We have $|S + |C| = k$ by assumption; (ii) since $(S,\emptyset,C)$ is valid, $\mathcal{G}(C) = \mathcal{G}(C \cup \{i'\}) \subseteq (S \cup C)$; (iii) since $(S,\emptyset,C)$ is valid, the inequality also holds for $(S \setminus \{i'\},\emptyset, C \cup \{i'\})$. Contradiction. Thus we have $\underline{\hat{w}_C} > \operatorname{max}_{i \in S} \hat{w}_i$.
    \item Suppose there does not exist $j \notin (S \cup \emptyset \cup C)$ such that $\underline{\hat{w}_C} > \hat{w}_j > \operatorname{max}_{i \in S} \hat{w}_i$. Denote $i^* \in S$ such that $\hat{w}_{i^*} = \operatorname{max}_{i \in S} \hat{w}_i$. We have $(S \setminus \{i^*\},\emptyset,C \cup \{i^*\})$ is a valid tuple, by verifying (i) $|C \cup \{i^*\}| = |C| + 1$, and $\operatorname{min} \{k - |S \setminus \{i^*\}|,d - |S \setminus \{i^*\}|\} = k - |S| + 1$. We have $|S| +|C| = k$ by assumption; (ii) since $(S,\emptyset,C)$ is valid, we have $\mathcal{G}(C)\subseteq (S \cup C)$. Thus $\mathcal{G}(C \cup \{i^*\}) = \mathcal{G}(C) \cup \{i^*\} \subseteq (S \cup C)$; (iii) since $(S,\emptyset,C)$ is a valid tuple, the inequality also holds for $(S \setminus \{i^*\},\emptyset,C \cup \{i^*\})$. Contradiction. Thus there exist some $j \notin (S \cup C)$ such that $\underline{\hat{w}_C} > \hat{w}_j > \operatorname{max}_{i \in S} \hat{w}_i$.
\end{itemize}

\noindent\textbf{3. The third scenario: $S \neq \emptyset$, $N \neq \emptyset$.} 

$\Leftarrow$: First we show that for any $i \in N$, $(S,N \setminus \{i\},C)$ is not valid. This is true, since it violates the second point (ii) given $\underline{\hat{w}_N} \geq \underline{\hat{w}_C}$, $i \notin \left(S \cup (N \setminus \{i\}) \cup C \right)$ and $\hat{w}_i \geq \underline{\hat{w}_C}$.

Next we show that for any $i' \in S$, $(S \setminus \{i'\},N,C \cup \{i'\})$ is not valid. Given some $j \notin (S \cup N \cup C)$ such that $\underline{\hat{w}_C} > \hat{w}_j > \operatorname{max}_{i \in S} \hat{w}_i$, it violates the second point (ii) $\hat{w}_j > \operatorname{max}_{i \in S} \hat{w}_i \geq \hat{w}_{i'} \geq \underline{\hat{w}_{C \cup \{i'\}}}$ and $j \notin (S \cup N \cup C)$. Then $(S \setminus \{i\},N,C \cup \{i\})$ is not valid as claimed.

$\Rightarrow$: First we show that $\underline{\hat{w}_N} \geq \underline{\hat{w}_C}$. Prove by contradiction. Suppose there exists some $j \in N$ such that $\hat{w}_j < \underline{\hat{w}_C}$. Then $(S, N \setminus \{j\},C)$ is a valid tuple by verifying that (i) $|C| = \operatorname{min} \{k - |S|, d - |N \setminus \{j\}| - |S|\} = k - |S|$; (ii) since $(S,N,C)$ is valid, we have $\mathcal{G}(C)\subseteq (S \cup N \cup C)$. Moreover, $j \notin \mathcal{G}(C)$ implies $\mathcal{G}(C) \subseteq (S \cup (N \setminus \{j\}) \cup C)$; (iii) the inequality holds since $(S,N,C)$ is a valid tuple by assumption. Contradiction. Thus we have $\underline{\hat{w}_N} \geq \underline{\hat{w}_C}$.

Next, we show that there exists some $j \notin (S \cup N \cup C)$ such that $\underline{\hat{w}_C} > \hat{w}_j > \operatorname{max}_{i \in S} \hat{w}_i$. Similarly, we prove by using contradiction twice. 

\begin{itemize}
    \item Suppose there exists some $i' \in S$ such that $\hat{w}_{i'} \geq \underline{\hat{w}_C}$. Then $(S \setminus \{i'\}, N, C \cup \{i'\})$ is valid, by verifying that (i) $|C \cup \{i'\}| = |C| + 1$, and $\operatorname{min}(k - |S \setminus \{i'\}|, d - |N| - |S \setminus \{i'\}|) = k - |S| + 1$. We have $|S + |C| = k$ by assumption; (ii) since $(S,N,C)$ is valid, we have $\mathcal{G}(C) = \mathcal{G}(C \cup \{i'\}) \subseteq (S \cup N \cup C)$; (iii) since $(S,N,C)$ is valid, the inequality also holds for $(S \setminus \{i\},N, C \cup \{i\})$. Contradiction. Thus we have $\underline{\hat{w}_C} > \operatorname{max}_{i \in S} \hat{w}_i$.
    \item Since we have proved that $\underline{\hat{w}_N} \geq \underline{\hat{w}_C}$, it is sufficient to show that there exists some $j \notin (S \cup C)$ such that $\underline{\hat{w}_C} > \hat{w}_j > \operatorname{max}_{i \in S} \hat{w}_i$. Prove by contradiction. Suppose such $j$ does not exist. Denote $i^* \in S$ such that $\hat{w}_{i^*} = \operatorname{max}_{i \in S} \hat{w}_i$. We have $(S \setminus \{i^*\},N,C \cup \{i^*\})$ is a valid tuple, by verifying (i) $|C \cup \{i^*\}| = |C| + 1$, and $\operatorname{min} \{k - |S \setminus \{i^*\}|,d - |N| - |S \setminus \{i^*\}|\} = k - |S| + 1$. We have $|S| +|C| = k$ by assumption; (ii) since $(S,N,C)$ is valid, we have $\mathcal{G}(C)\subseteq (S \cup N \cup C)$. Thus $\mathcal{G}(C \cup \{i^*\}) = \mathcal{G}(C) \cup \{i^*\} \subseteq (S \cup N \cup C)$; (iii) since $(S,N,C)$ is a valid tuple, the inequality also holds for $(S \setminus \{i^*\},N,C \cup \{i^*\})$. Contradiction. Thus there exist some $j \notin (S \cup N \cup C)$ such that $\underline{\hat{w}_C} > \hat{w}_j > \operatorname{max}_{i \in S} \hat{w}_i$.
\end{itemize}

\end{proof}

\subsection{Proof of Proposition \ref{prop:knapsack-cuts-generation}} \label{app:knapsack-cuts-generation}

\knapsackcuts*

\begin{proof}
    We separately give proofs on the two scenarios presented in Proposition (\ref{prop:knapsack-cuts-generation}). 

    \noindent \textbf{1. The first scenario:} $|\supp(\bm{y})| < k$
    
    In this case, we have $|S \cup C| = |\supp(\bm{y})| < k$, and $ N = [\![d]\!] \setminus (S \cup C) = [\![d]\!] \setminus \supp(\bm{y})$ based on validity. From Proposition (\ref{prop:minimal-tuple}), we know that if $|N| > d -k$, then $(S,N,C)$ is minimal if and only if $S = \emptyset$. Thus we can construct the minimal tuple as $(S,N,C) = (\emptyset,[\![d]\!] \backslash \supp(\bm{y}),\supp(\bm{y}))$. 

    \noindent \textbf{2. The second scenario:} $|\supp (\bm{y})| = k$

    In this case, we have $|S \cup C| = |\supp (\bm{y})| = k$, and $|N| \leq d-k$ based on validity. 

    \begin{enumerate}
        \item First, we construct a minimal tuple with $C = \emptyset$. Thus we have $S = \supp (\bm{y})$. From Proposition (\ref{prop:minimal-tuple}), we know that $(\supp(\bm{y}),N,\emptyset)$ is minimal if and only if $N = \emptyset$, and $\operatorname{max}_{i \in \supp(\bm{y})} \Tilde{w}_i < \Tilde{w}_{[1]}$. Therefore, $(\supp(\bm{y}),\emptyset,\emptyset)$ is minimal if and only $\operatorname{max}_{i \in \supp(\bm{y})} \Tilde{w}_i < \Tilde{w}_{[1]}$, which is equivalent to $B_1(\bm{y}) \neq \emptyset$ since $\tilde{\bm{w}}$ is in decreasing order.
        \item Second, we construct minimal tuple with $C \neq \emptyset$. Denote $i^* \in C \subseteq \supp(\bm{y})$ the index with respect to the smallest component in $\Tilde{w}_C$, i.e., $\Tilde{w}_{i^*} = \underline{\Tilde{w}_C}$. We have two subcases:
    \begin{itemize}
        \item  If $i^* \in A_m(\bm{y})$, then $N \supseteq \comp(\supp(\bm{y}))$ to ensure validity. Based on Proposition (\ref{prop:minimal-tuple}), we have minimal tuple as $(S,N,C) = (\emptyset,\comp(\supp(\bm{y})),\supp(\bm{y}))$.
        \item If $i^* \in A_{\ell}(\bm{y})$ for some $\ell = 1, \ldots, m-1$, similarly, we can construct minimal SCG-tuple as $(S,N,C) = (\cup_{j= \ell+1}^m A_j(\bm{y}), ~ \cup_{j=1}^l B_j(\bm{y}), ~ \cup_{j=1}^l A_j(\bm{y}))$.
    \end{itemize}
    \end{enumerate}
\end{proof}

\subsection{Proof of bounds for potential screening ability} \label{app:potential-screening-ability}

\begin{proof}
    For any screening cut, it removes all the feasible points that satisfy $z_i = 1,$ for all $i \in S$ and $z_j = 0,$ for all $j \in N$. If $|N| > d - k$, then based on validity, we have $|S| + |N| + |C| = d$ and thus $(S\cup N \cup C) = [\![d]\!]$. The remaining index set is $C$ and we can choose up to $|C|$ one's within $C$. We have at most $\sum_{i=0}^{|C|}{|C| \choose i} = \sum_{i=0}^{|C|}{d - |S \cup N|\choose i}$ eliminated points. Similarly, if $|N| \leq d-k$, we have $|S| + |C| = k$ and $(S \cup N \cup C) \subseteq [\![d]\!]$. The remaining index set is $[\![d]\!] \backslash (S \cup N)$ and we can choose up to $|C|$ one's. Thus we have at most $\sum_{i=0}^{|C|}{d - |S \cup N|\choose i}$ eliminated points.
\end{proof}

\subsection{Proofs in Section~\ref{sec:implementation-technique}} \label{app:algorithm}

\subsubsection{Inclusive Cuts Generation} \label{app:inclusive-cuts}

\inclusive*
\begin{proof}
    We show two results in Claim~\ref{clm:inclusive} separately.  \\
    \textbf{Proof on the first result.} One direction is obvious and we suppose $\mathcal{T}_i \neq \emptyset$. Pick any $\bm{y}' \in \mathcal{T}_i$. Given that $\Tilde{\bm{w}}$ is in decreasing order and $\bm{y}' \in K(\Tilde{\bm{w}})$, we have $\Tilde{\bm{w}}^T \bm{v} \leq \Tilde{\bm{w}}^T \bm{y}' < c(\Tilde{\bm{w}})$, since $\bm{v} \in \{0,1\}^d$ and $\supp(\bm{v}) = [\![i - (k-1),i]\!]$ as assumed. Thus $\bm{v} \in K(\Tilde{\bm{w}})$.\\
    \textbf{Proof on the second result.} Given $\mathcal{T}_{i'} \neq \emptyset$, it is sufficient to show that $\mathcal{T}_{i'+1} \neq \emptyset$. Based on the first part, we have $\bm{v}_{i'} \in K(\Tilde{\bm{w}})$ and $\supp(\bm{v}_{i'}) = [\![i' - (k - 1),i']\!]$. We construct $\bm{v}_{i'+1} \in \{0,1\}^d$ and $\supp(\bm{v}_{i'+1}) = [\![i'+1 - (k-1),i'+1]\!]$. It is obvious that $\bm{v}^T_{i'+1} \Tilde{\bm{w}} < \bm{v}^T_{i'} \Tilde{\bm{w}} < c(\Tilde{\bm{w}})$. Then $\bm{v}_{i'+1} \in K(\Tilde{\bm{w}})$ and we are done.
\end{proof}

\subsubsection{Exclusive Cuts Generation} \label{app:exclusive-cuts} 

\exclusive*
\begin{proof}
    We show two results in Claim~\ref{clm:exclusive} separately. \\
    \textbf{Proof on the first result.} One direction is obvious and we suppose $\mathcal{J}_i \neq \emptyset$. Pick any $\bm{y}' \in \mathcal{J}_i$. Given that $\Tilde{\bm{w}}$ is in decreasing order and $\bm{y}' \in K(\Tilde{\bm{w}})$, we have $\Tilde{\bm{w}}^T \bm{v} \leq \Tilde{\bm{w}}^T \bm{y}' < c(\Tilde{\bm{w}})$, since $\bm{v} \in \{0,1\}^d$ and $\supp(\bm{v}) = [\![i]\!] \cup [\![d - (k-i-1),d]\!]$ as assumed. Thus $\bm{v} \in K(\Tilde{\bm{w}})$. \\
    \textbf{Proof on the second result.} Given $\mathcal{J}_{i'} \neq \emptyset$, it is sufficient to show that $\mathcal{J}_{i'-1} \neq \emptyset$. Based on the first part, we have $\bm{v}_{i'} \in K(\Tilde{\bm{w}})$ and $\supp(\bm{v}_{i'}) = [\![i']\!] \cup [\![d - (k-i'-1),d]\!]$. We construct $\bm{v}_{i'-1} \in \{0,1\}^d$ and $\supp(\bm{v}_{i'-1}) = [\![i'-1]\!] \cup [\![d - (k-i'),d]\!]$. It is obvious that $\bm{v}^T_{i'-1} \Tilde{\bm{w}} < \bm{v}^T_{i'} \Tilde{\bm{w}} < c(\Tilde{\bm{w}})$. Then $\bm{v}_{i'-1} \in K(\Tilde{\bm{w}})$ and we are done.
\end{proof}

\subsubsection{SCG Algorithm Design}

\generateall*

\begin{proof}  
    We show results on inclusive cuts generation and exclusive cuts generation separately. \\
    \textbf{Inclusive cuts.} Based on previous analysis, for each $\bm{y} \in K(\Tilde{\bm{w}})$ with $|\supp(\bm{y})| < k$, the minimal inclusive cut has length $|[\![d]\!] \setminus \supp(\bm{y})| > d-k \geq \sharp_{\tt len}$. Therefore, we have $\bm{y} \in \mathcal{O}$ to generate such inclusive cuts.

    For any $\bm{y} \in \mathcal{O}$ that can generate minimal inclusive cut with length at most $\sharp_{\tt len}$, it is sufficient to prove that $\bm{y} \in \mathcal{T}_{i'}$ for some $i' \in [\![k+1,d]\!]$ and $\mathcal{T}_{i'}$ shall be found in STEP 2 of Algorithm \ref{alg:alg-framework}.
    
    Given a maximum length $\sharp_{\tt len}$, we have $|\comp(\supp(\bm{y}))| = i' - k \leq \sharp_{\tt len}$, which is equivalent to $i' \leq k + \sharp_{\tt l}$. Therefore, we have the searching index $s$ initialized in STEP 1 satisfies $s \leq i'$. Since $\sharp_{\tt max} = + \infty$, the Algorithm \ref{alg:alg-framework} will not stop until $s = \operatorname{min}\{k + \sharp_{\tt l},d\}$. Thus we have $s = i'$ before termination. 
    
    \noindent \textbf{Exclusive cuts.} Based on previous analysis, for each $\bm{y} \in K(\Tilde{\bm{w}})$ with $|\supp(\bm{y})|  = k$ and $B_1(\bm{y}) \neq \emptyset$, the minimal exclusive cut has length $|\supp(\bm{y})| =k > \sharp_{\tt len}$. Therefore, we have $\bm{y} \in \mathcal{O}'$ to generate such exclusive cuts.

    For any $\bm{y} \in \mathcal{O}'$ that can generate minimal exclusive cut with length at most $\sharp_{\tt len}$, it is sufficient to prove that $\bm{y} \in \mathcal{T}_{i'}'$ for some $i' \in [\![1,k-1]\!]$ and $\mathcal{T}_{i'}'$ shall be found in STEP 2 of Algorithm \ref{alg:alg-framework}.

    Given a maximum length $\sharp_{\tt l}$, we have $|\cup_{j=2}^m A_j(\bm{y}')| = k - i' \leq \sharp_{\tt len}$, which is equivalent to $i' \geq k - \sharp_{\tt l}$. Therefore, in STEP 1, we have the initial searching index $s \geq i'$. Since $\sharp_{\tt max}  = + \infty$, the Algorithm \ref{alg:alg-framework} will not stop until $s = \operatorname{max}\{k - \sharp_{\tt l},1\}$. Thus we have $s = i'$ before termination.
\end{proof}

Here, we present the pseudo-code (Algorithm~\ref{alg:recursive-iteration}) for the recursion used in STEP 2 of Algorithm~\ref{alg:alg-framework}. Given a search index $s \in \texttt{SR}_{\tt inc}$ or $\texttt{SR}_{\tt exc}$, the computational complexity of enumerating $\mathcal{T}_s$ or $\mathcal{T}_s'$ is $O\big((k-1)^2 \cdot{s-1 \choose k-1}\big)$ or $O\big((k-s)^2 \cdot {d- s -1 \choose k-s}\big)$ for inclusive or exclusive cuts generation. \\

\begin{algorithm} \label{alg:recursive-iteration}
\caption{Recursion in STEP 2 of Algorithm~\ref{alg:alg-framework}}

\KwIn{Sparsity level $k$, sorted $\tilde{\bm{w}}$,
current search index $s$ in STEP 2. 
}

Set weight limit $\texttt{WL} = c(\tilde{\bm{w}}) - \tilde{\bm{w}}_s$ for inclusive cuts, $\texttt{WL} = c(\tilde{\bm{w}}) - \sum_{j = 1}^s \tilde{w}_j$ for exclusive cuts. 

Set searching range $\texttt{SR} = [\![s-1]\!]$ for inclusive cuts, $\texttt{SR} = [\![s + 2, d]\!]$ for exclusive cuts. 

Set select number $\texttt{SN} = k - 1$ for inclusive cuts, $\texttt{SN} = k - s$ for exclusive cuts.

Initialized target partition subset $\mathcal{T}_s = \emptyset$ and temporary enumerating support 
$\texttt{T}_{\texttt{supp}} = \emptyset$.

\SetKwFunction{FRecursive}{\tt RI}
\SetKwProg{Fn}{Function}{:}{}
\SetKw{WL}{\texttt{WL}}
\SetKw{SR}{\texttt{SR}}
\SetKw{SN}{\texttt{SN}}
\SetKwFunction{FSelectAll}{SelectAll}
\SetKwProg{Fn}{Function}{:}{}

\Fn{   \FRecursive{\WL, $\tilde{\bm{w}}$, \SR, \SN,  $\mathcal{T}_s$, $\texttt{T}_{\texttt{supp}}$  }     }{
    \For{index $i \in [\![|\SR| - \SN +1]\!]$}{
        Select $i$th smallest index in $\texttt{SR}$, i.e., index $j = \texttt{SR}[i]$\;
        Compute $\texttt{min-weight} = \tilde{w}_j + \sum_{t=|\texttt{SR}|-\texttt{SN}+2}^{|\tt SR|} \tilde{w}_{\texttt{SR}[t]}$\;
        
        \If{$\texttt{min-weight} < \WL$}{ 
            Add index $j$ to current enumerating support, i.e., $\texttt{T}_{\texttt{supp}} = \texttt{T}_{\texttt{supp}} \cup \{j\}$\;
            
            \eIf{$\texttt{SN} = 1$}{
                Add current temporary enumerating support $\texttt{T}_{\texttt{supp}}$ into $\mathcal{T}_s$, i.e., $\mathcal{T}_s = \mathcal{T}_s \cup \{\texttt{T}_{\texttt{supp}} \}$ \;
            }{
                Update weight limit $\WL = \WL - \tilde{w}_j$\;
                Remove all the indices that are smaller than or equal to $j = \texttt{SR}[i]$ from $\texttt{SR}$, i.e., $\texttt{SR} = \{\texttt{SR}[i+1], \ldots, \texttt{SR}[|\SR|]\}$\;
                Update current $\SN = \SN-1$\;
                \FRecursive{$\WL$, $\tilde{\bm{w}}$, $\SR$, $\SN$, $\mathcal{T}_s$, $\texttt{T}_{\texttt{supp}}$}\;
            }
            
            Remove index $j$ from temporary support set $\texttt{T}_{\texttt{supp}}$, i.e.,  $\texttt{T}_{\texttt{supp}} = \texttt{T}_{\texttt{supp}} \setminus \{j\}$  \;
        }
    }
    \Return target partition set $\mathcal{T}_s$\;
}
\end{algorithm}

\begin{remark}
    The Algorithm \ref{alg:recursive-iteration}  generates all possible supports of \texttt{SN} elements from index set \texttt{SR} that satisfy outer if condition presented in Algorithm~\ref{alg:recursive-iteration}. If outer if-condition is satisfied for every possible support, Algorithm \ref{alg:recursive-iteration} is equivalent to brute-force enumeration. Specifically, there are ${|\texttt{SR}| \choose \texttt{SN}}$ possible support sets, and each support requires at most $O(\texttt{SN}^2)$ elementary algebraic operations in our recursion Algorithm \ref{alg:recursive-iteration}. Therefore, the total computational complexity is $O\big( \texttt{SN}^2 \cdot  {|\texttt{SR}| \choose \texttt{SN}}\big)$. The main difference between our recursion Algorithm \ref{alg:recursive-iteration} and brute-force enumeration lies in identifying whether index $j$ presented in Algorithm~\ref{alg:recursive-iteration} can be added in any possible support.
\end{remark}

\section{Tables in Section \ref{sec:numerical-discussion}} \label{app:cuts-table}

\noindent\textbf{Summary of pre-solving time}. We report the pre-solving time for selecting inclusive and exclusive cuts by SCG and SSR in Table \ref{tbl:pre-solving-time}. The first column indicates the employed methods. The second to fourth columns report the pre-solving time on synthetic datasets with dimension $d \in \{100,300,500\}$, respectively. Each cell contains a pair of format $(\leq t^{\texttt{ub}}_{\text{0.5}}, ~\leq t^{\texttt{ub}}_{3.5} )$. 
Here, given a fixed parameter setting/configuration and a fixed sequence of conducted ridge parameters $\{\gamma_i\}_{i \in I}$, $t^{\texttt{ub}}_{\text{0.5}}$ and $t^{\texttt{ub}}_{\text{3.5}}$ denote the upper bounds of pre-solving time among all distinct $\gamma$s for $\text{SNR} \in \{0.5,~3.5\}$, respectively, i.e., 
\begin{align*}
    t^{\texttt{ub}}_{\text{0.5}} := \max_{i \in I} ~ t^{\texttt{pre}}_{\star} (\gamma_i; d, \text{SNR} = 0.5) \quad \text{and} \quad t^{\texttt{ub}}_{\text{3.5}} := \max_{i \in I} ~ t^{\texttt{pre}}_{\star} (\gamma_i; d, \text{SNR} = 3.5) ~~
\end{align*}
with $\star \in \{\text{SCG, SSR}\}$ and $t^{\texttt{pre}}_{\star} (\gamma_i; d, \text{SNR})$ the pre-processing time for variable screening method $\star$ with dataset parameter configuration $(d, \text{SNR})$. The last two columns provide the pre-solving time on two real datasets: UJIndoorLoc and CNAE. Each cell contains the pair $(\leq t^{\texttt{ub}}_{\texttt{reduced}}, ~\leq t^{\texttt{ub}}_{\texttt{all}} )$, where $t^{\texttt{ub}}_{\texttt{reduced}}$ and $t^{\texttt{ub}}_{\texttt{all}}$ represents the upper bounds of pre-solving time among all distinct $\gamma$s on the reduced sample dataset and whole dataset, respectively. \\

\begin{table}[h!]
\centering
\caption{Pre-solving time for SCG and SSR(s)}
\label{tbl:pre-solving-time}
\vskip 0.15in
\begin{center}
\begin{small}
\begin{sc}
\begin{tabular}{l|c|c|c|c|c}
\toprule
Methods & $d = 100$ & $d = 300$ & $d = 500$ & UJIndoorLoc & CNAE \\
\midrule
SCG   &  ($\leq 0.21$, $\leq 0.24$) &  ($\leq 0.49$, $\leq 0.69$) & ($\leq 0.76$, $\leq 1.22$)& ($\leq 4.73$, $\leq 27.6$) & ($\leq 0.7$, $\leq 3.2$)\\
\midrule
SSR   & ($\leq 0.39$, $\leq 0.16$) &  ($\leq 1.17$, $\leq 0.72$) & ($\leq 2.76$, $\leq 1.54$) & ($\leq 5.12$, $\leq 28.4$)& ($\leq 2.1$, $\leq 4.3$) \\
\bottomrule
\end{tabular}
\end{sc}
\end{small}
\end{center}
\vskip -0.1in
\end{table}

\noindent \textbf{Notation used in cut characteristic tables.} Before reporting numerical results on cut characteristics (Table~\ref{tbl:cuts_100_0.5}, \ref{tbl:cuts_100_3.5}, \ref{tbl:cuts_300_0.5}, \ref{tbl:cuts_300_3.5}, \ref{tbl:cuts_500_0.5}, \ref{tbl:cuts_500_3.5}), let us first introduce notation that will be used in these tables. We use $n_{\texttt{inc}}$ and $n_{\texttt{exc}}$ to denote the average number of inclusive and exclusive cuts generated, respectively. The average number of binary variables involved in each inclusive and exclusive cut is represented as  $l_{\texttt{inc}}$ and $l_{\texttt{exc}}$. We use $(\texttt{NA}_{\texttt{inc}},\texttt{NA}_{\texttt{exc}})$ (a shorthand for ``None Added'') to denote the number of datasets where no inclusive cut or no exclusive cut is added.

\clearpage
\noindent\textbf{Synthetic datasets - $d = 100$}

\begin{table}[h!]
\centering
\caption{$d=100,~\text{SNR} = 0.5$}
\label{tbl:cuts_100_0.5}
\vskip 0.15in
\begin{center}
\begin{small}
\begin{sc}
\begin{tabular}{l|c|c|c|c}
\toprule
$\gamma$ & Methods & $(n_{\texttt{inc}},n_{\texttt{exc}})$ & $(l_{\texttt{inc}},l_{\texttt{exc}})$ & $(\texttt{na}_{\texttt{inc}},\texttt{na}_{\texttt{exc}})$ \\
\midrule
2.2    & \makecell{SCG \\ SSR}  &  \makecell{(3.9,~89) \\ (2.9,~71.9)}  & \makecell{(1.27,~1.19) \\ - } & \makecell{(1,~0) \\ (1,~0)} \\
\midrule
2   & \makecell{SCG \\ SSR}   & \makecell{(3, 89.5)\\(2.6, 70.7)} &  \makecell{(1.09, 1.21) \\ -} & \makecell{(1, 0) \\ (1, 0)} \\ 
\midrule
1.8 & \makecell{SCG \\ SSR} & \makecell{(3.3, 89.2) \\ (2.1, 58.3)} & \makecell{(1.4, 1.34) \\ -} & \makecell{(2, 0) \\ (4, 1)} \\
\midrule
1.6 & \makecell{SCG \\ SSR} & \makecell{(2.3, 89.1) \\ (1.9, 60.3)} & \makecell{(1.15, 1.32) \\ -} & \makecell{(0, 0) \\ (1, 0)} \\
\midrule
1.4 & \makecell{SCG \\ SSR} & \makecell{(1.8, 89.2) \\ (1, 52.9)} & \makecell{(1.28, 1.4) \\ -} & \makecell{(4, 0) \\ (4, 0)} \\

\bottomrule
\end{tabular}
\end{sc}
\end{small}
\end{center}
\vskip -0.1in
\end{table}

\begin{table}[h!]
\centering
\caption{$d= 100,~\text{SNR} = 3.5$}
\label{tbl:cuts_100_3.5}
\vskip 0.15in
\begin{center}
\begin{small}
\begin{sc}
\begin{tabular}{l|c|c|c|c}
\toprule
$\gamma$ & Methods & $(n_{\texttt{inc}},n_{\texttt{exc}})$ & $(l_{\texttt{inc}},l_{\texttt{exc}})$ & $(\texttt{na}_{\texttt{inc}},\texttt{na}_{\texttt{exc}})$ \\
\midrule
2.2    & \makecell{SCG \\ SSR}  &  \makecell{(5.1,~86.5) \\ (4.8,~78.9)}  & \makecell{(1.08,~1.08) \\ - } & \makecell{(0,~0) \\ (0,~0)} \\
\midrule
2   & \makecell{SCG \\ SSR}   & \makecell{(4.9, 88.7)\\(3.7,~ 77.5)} &  \makecell{(1.25, 1.16) \\ -} & \makecell{(0, 0) \\ (0, 0)} \\ 
\midrule
1.8 & \makecell{SCG \\ SSR} & \makecell{(5, 87.3) \\ (2.8, 73.8)} & \makecell{(1.35, 1.15) \\ -} & \makecell{(0, 0) \\ (0, 0)} \\
\midrule
1.6 & \makecell{SCG \\ SSR} & \makecell{(4.5, 88.5) \\ (3.1, 70.3)} & \makecell{(1.26, 1.2) \\ -} & \makecell{(0, 0) \\ (1, 0)} \\
\midrule
1.4 & \makecell{SCG \\ SSR} & \makecell{(3.6, 88.7) \\ (3.2, 67.7)} & \makecell{(1.13, 1.24) \\ -} & \makecell{(0, 0) \\ (1, 0)} \\

\bottomrule
\end{tabular}
\end{sc}
\end{small}
\end{center}
\vskip -0.1in
\end{table}

\vspace{1cm}

\clearpage
\noindent\textbf{Synthetic datasets - $d = 300$}

\vspace{1cm}

\begin{table}[h!]
\centering
\caption{$d = 300,~\text{SNR} = 0.5$}
\label{tbl:cuts_300_0.5}
\vskip 0.15in
\begin{center}
\begin{small}
\begin{sc}
\begin{tabular}{l|c|c|c|c}
\toprule
$\gamma$ & Methods & $(n_{\texttt{inc}},n_{\texttt{exc}})$ & $(l_{\texttt{inc}},l_{\texttt{exc}})$ & $(\texttt{na}_{\texttt{inc}},\texttt{na}_{\texttt{exc}})$ \\
\midrule
2.2    & \makecell{SCG \\ SSR}  &  \makecell{(2.6,~290) \\ (0.9,~228)}  & \makecell{(1.55,~1.21) \\ - } & \makecell{(2,~0) \\ (4,~0)} \\
\midrule
2   & \makecell{SCG \\ SSR}   & \makecell{(1.1, 290)\\(0.6, 217.1)} &  \makecell{(1.5, 1.25) \\ -} & \makecell{(4, 0) \\ (6, 0)} \\ 
\midrule
1.8 & \makecell{SCG \\ SSR} & \makecell{(1.6, 290) \\ (0.6, 218)} & \makecell{(1.52, 1.25) \\ -} & \makecell{(3, 0) \\ (5, 0)} \\
\midrule
1.6 & \makecell{SCG \\ SSR} & \makecell{(0.7, 290) \\ (0.2, 138.4)} & \makecell{(1.7, 1.52) \\ -} & \makecell{(5, 0) \\ (8, 2)} \\
\midrule
1.4 & \makecell{SCG \\ SSR} & \makecell{(0.6, 290) \\ (0.2, 140.5)} & \makecell{(1.67, 1.52) \\ -} & \makecell{(7, 0) \\ (9, 1)} \\

\bottomrule
\end{tabular}
\end{sc}
\end{small}
\end{center}
\vskip -0.1in
\end{table}

\vspace{1cm}

\begin{table}[h!]
\centering
\caption{$d = 300,~\text{SNR} = 3.5$}
\label{tbl:cuts_300_3.5}
\vskip 0.15in
\begin{center}
\begin{small}
\begin{sc}
\begin{tabular}{l|c|c|c|c}
\toprule
$\gamma$ & Methods & $(n_{\texttt{inc}},n_{\texttt{exc}})$ & $(l_{\texttt{inc}},l_{\texttt{exc}})$ & $(\texttt{na}_{\texttt{inc}},\texttt{na}_{\texttt{exc}})$ \\
\midrule
2.2    & \makecell{SCG \\ SSR}  &  \makecell{(2.3,~290) \\ (1.9,~255.7)}  & \makecell{(1.15,~1.12) \\ - } & \makecell{(0,~0) \\ (1,~0)} \\
\midrule
2   & \makecell{SCG \\ SSR}   & \makecell{(3.2, 290)\\(1.5,~ 258.6)} &  \makecell{(1.48, 1.11) \\ -} & \makecell{(0, 0) \\ (2, 0)} \\ 
\midrule
1.8 & \makecell{SCG \\ SSR} & \makecell{(3.5, 290) \\ (2, 232.7)} & \makecell{(1.36, 1.20) \\ -} & \makecell{(1, 0) \\ (1, 0)} \\
\midrule
1.6 & \makecell{SCG \\ SSR} & \makecell{(2.9, 290) \\ (2.1, 225.6)} & \makecell{(1.31, 1.22) \\ -} & \makecell{(0, 0) \\ (2, 0)} \\
\midrule
1.4 & \makecell{SCG \\ SSR} & \makecell{(1.2, 290) \\ (0.5, 177.9)} & \makecell{(1.52, 1.39) \\ -} & \makecell{(3, 0) \\ (6, 2)} \\

\bottomrule
\end{tabular}
\end{sc}
\end{small}
\end{center}
\vskip -0.1in
\end{table}

\clearpage
\noindent\textbf{Synthetic datasets - $d = 500$}

\vspace{1cm}

\begin{table}[h!]
\centering
\caption{$d = 500,~\text{SNR} = 0.5$}
\label{tbl:cuts_500_0.5}
\vskip 0.15in
\begin{center}
\begin{small}
\begin{sc}
\begin{tabular}{l|c|c|c|c}
\toprule
$\gamma$ & Methods & $(n_{\texttt{inc}},n_{\texttt{exc}})$ & $(l_{\texttt{inc}},l_{\texttt{exc}})$ & $(\texttt{na}_{\texttt{inc}},\texttt{na}_{\texttt{exc}})$ \\
\midrule
2.2    & \makecell{SCG \\ SSR}  &  \makecell{(1.1,~490) \\ (1.1,~418.2)}  & \makecell{(1,~1.15) \\ - } & \makecell{(5,~0) \\ (5,~0)} \\
\midrule
2   & \makecell{SCG \\ SSR}   & \makecell{(1.6, 490)\\(0.6, 390.3)} &  \makecell{(1.38, 1.2) \\ -} & \makecell{(4, 0) \\ (5, 0)} \\ 
\midrule
1.8 & \makecell{SCG \\ SSR} & \makecell{(1, 490) \\ (0.1, 288.1)} & \makecell{(1.92, 1.41) \\ -} & \makecell{(7, 0) \\ (9, 1)} \\
\midrule
1.6 & \makecell{SCG \\ SSR} & \makecell{(0.3, 490) \\ (0.2, 245.9)} & \makecell{(1.33, 1.5) \\ -} & \makecell{(7, 0) \\ (8, 2)} \\
\midrule
1.4 & \makecell{SCG \\ SSR} & \makecell{(0.3, 490) \\ (0, 127.5)} & \makecell{(2, 1.74) \\ -} & \makecell{(8, 0) \\ (10, 6)} \\

\bottomrule
\end{tabular}
\end{sc}
\end{small}
\end{center}
\vskip -0.1in
\end{table}

\vspace{1cm}

\begin{table}[h!]
\centering
\caption{$d = 500,~\text{SNR} = 3.5$}
\label{tbl:cuts_500_3.5}
\vskip 0.15in
\begin{center}
\begin{small}
\begin{sc}
\begin{tabular}{l|c|c|c|c}
\toprule
$\gamma$ & Methods & $(n_{\texttt{inc}},n_{\texttt{exc}})$ & $(l_{\texttt{inc}},l_{\texttt{exc}})$ & $(\texttt{na}_{\texttt{inc}},\texttt{na}_{\texttt{exc}})$ \\
\midrule
2.2    & \makecell{SCG \\ SSR}  &  \makecell{(2,~490) \\ (1.3,~440.1)}  & \makecell{(1.28,~1.1) \\ - } & \makecell{(1,~0) \\ (2,~0)} \\
\midrule
2   & \makecell{SCG \\ SSR}   & \makecell{(2.1, 489.9)\\(1.5,~ 441.9)} &  \makecell{(1.32, 1.1) \\ -} & \makecell{(2, 0) \\ (3, 0)} \\ 
\midrule
1.8 & \makecell{SCG \\ SSR} & \makecell{(2.3, 490) \\ (1.4, 383.3)} & \makecell{(1.34, 1.22) \\ -} & \makecell{(2, 0) \\ (3, 1)} \\
\midrule
1.6 & \makecell{SCG \\ SSR} & \makecell{(1.2, 490) \\ (0.7, 378.1)} & \makecell{(1.4, 1.23) \\ -} & \makecell{(3, 0) \\ (5, 0)} \\
\midrule
1.4 & \makecell{SCG \\ SSR} & \makecell{(1, 490) \\ (0.3, 230.7)} & \makecell{(1.67, 1.53) \\ -} & \makecell{(4, 0) \\ (8, 3)} \\

\bottomrule
\end{tabular}
\end{sc}
\end{small}
\end{center}
\vskip -0.1in
\end{table}

\clearpage

\noindent\textbf{Real dataset - UJIndoorLoc}

\vspace{1cm}

\begin{table}[h!]
\centering
\caption{UJIndoorLoc}
\label{tbl:UJIndoor}
\vskip 0.15in
\begin{center}
\begin{small}
\begin{sc}
\begin{tabular}{l|c|c|c|c|c}
\toprule
$\gamma$ & Methods & $(n_{\texttt{inc}}^{\texttt{all}},n_{\texttt{exc}}^{\texttt{all}})$ & $(l_{\texttt{inc}}^{\texttt{all}},l_{\texttt{exc}}^{\texttt{all}})$ & $(n_{\texttt{inc}}^{\texttt{red}},n_{\texttt{exc}}^{\texttt{red}})$ & $(l_{\texttt{inc}}^{\texttt{red}},l_{\texttt{exc}}^{\texttt{red}})$\\
\midrule
14000    & \makecell{SCG \\ SSR}  &  \makecell{(0,~510) \\ (0, 439)}  & \makecell{(\texttt{NA},~1.14) \\ - } &  \makecell{(0,~510) \\ (0, 459)}  & \makecell{(\texttt{NA},~1.1) \\ - } \\
\midrule
13000   & \makecell{SCG \\ SSR}   & \makecell{(0, 510)\\(0, 439)} &  \makecell{(\texttt{NA}, 1.14) \\ -} &  \makecell{(0,~510) \\ (0, 69)}  & \makecell{(\texttt{NA},~1.86) \\ - } \\ 
\midrule
12000 & \makecell{SCG \\ SSR} & \makecell{(0, 510) \\ (0, 429)} & \makecell{(\texttt{NA}, 1.15) \\ -}  &  \makecell{(0,~510) \\ (0, 51)}  & \makecell{(\texttt{NA},~1.9) \\ - }\\
\midrule
11000& \makecell{SCG \\ SSR} & \makecell{(0, 510) \\ (0, 98)} & \makecell{(\texttt{NA}, 1.80) \\ -} &  \makecell{(0,~510) \\ (0, 456)}  & \makecell{(\texttt{NA},~1.10) \\ - } \\
\midrule
10000 & \makecell{SCG \\ SSR} & \makecell{(0, 510) \\ (0, 83)} & \makecell{(\texttt{NA}, 1.83) \\ -} &  \makecell{(0,~510) \\ (0, 459)}  & \makecell{(\texttt{NA},~1.1) \\ - } \\
\midrule
9000 & \makecell{SCG \\ SSR} & \makecell{(0, 510) \\ (0, 358)} & \makecell{(\texttt{NA}, 1.29) \\ -} &  \makecell{(0,~510) \\ (0, 461)}  & \makecell{(\texttt{NA},~1.09) \\ - } \\
\midrule
8000 & \makecell{SCG \\ SSR} & \makecell{(0, 510) \\ (0, 359)} & \makecell{(\texttt{NA}, 1.29) \\ -} &  \makecell{(0,~510) \\ (0, 459)}  & \makecell{(\texttt{NA},~1.1) \\ - } \\
\midrule
7000 & \makecell{SCG \\ SSR} & \makecell{(0, 510) \\ (0, 114)} & \makecell{(\texttt{NA}, 1.77) \\ -} &  \makecell{(0,~510) \\ (0, 452)}  & \makecell{(\texttt{NA},~1.11) \\ - } \\
\midrule
6000 & \makecell{SCG \\ SSR} & \makecell{(0, 510) \\ (0, 65)} & \makecell{(\texttt{NA}, 1.87) \\ -} &  \makecell{(0,~510) \\ (0, 120)}  & \makecell{(\texttt{NA},~1.76) \\ - } \\
\midrule
5000 & \makecell{SCG \\ SSR} & \makecell{(0, 510) \\ (0, 0)} & \makecell{(\texttt{NA}, 2) \\ -} &  \makecell{(0,~510) \\ (0, 80)}  & \makecell{(\texttt{NA},~1.84) \\ - } \\
\midrule
4000 & \makecell{SCG \\ SSR} & \makecell{(0, 510) \\ (0, 0)} & \makecell{(\texttt{NA}, 2) \\ -} &  \makecell{(0,~510) \\ (0, 38)}  & \makecell{(\texttt{NA},~1.92) \\ - } \\
\midrule
3000 & \makecell{SCG \\ SSR} & \makecell{(0, 510) \\ (0, 0)} & \makecell{(\texttt{NA}, 2) \\ -} &  \makecell{(0,~510) \\ (0, 9)}  & \makecell{(\texttt{NA},~1.98) \\ - }\\

\bottomrule
\end{tabular}
\end{sc}
\end{small}
\end{center}
\vskip -0.1in
\end{table}

\clearpage

\noindent\textbf{Real dataset - CNAE}

\vspace{1cm}

\begin{table}[h!]
\centering
\caption{CNAE}
\label{tbl:CNAE}
\vskip 0.15in
\begin{center}
\begin{small}
\begin{sc}
\begin{tabular}{l|c|c|c|c|c}
\toprule
$\gamma$ & Methods & $(n_{\texttt{inc}}^{\texttt{all}},n_{\texttt{exc}}^{\texttt{all}})$ & $(l_{\texttt{inc}}^{\texttt{all}},l_{\texttt{exc}}^{\texttt{all}})$ & $(n_{\texttt{inc}}^{\texttt{red}},n_{\texttt{exc}}^{\texttt{red}})$ & $(l_{\texttt{inc}}^{\texttt{red}},l_{\texttt{exc}}^{\texttt{red}})$\\
\midrule
0.1    & \makecell{SCG \\ SSR}  &  \makecell{(3,~846) \\ (3, 831)}  & \makecell{(1,~1.01) \\ - } &  \makecell{(4,~846) \\ (4, 830)}  & \makecell{(1,~1.02) \\ - } \\
\midrule
0.09   & \makecell{SCG \\ SSR}   & \makecell{(4, 846)\\(2, 828)} &  \makecell{(1.5, 1.02) \\ -} &  \makecell{(4,~846) \\ (4, 821)}  & \makecell{(1,~1.03) \\ - } \\ 
\midrule
0.08 & \makecell{SCG \\ SSR} & \makecell{(4, 846) \\ (1, 780)} & \makecell{(1.75, 1.07) \\ -}  &  \makecell{(4,~846) \\ (2, 796)}  & \makecell{(1.5,~1.06) \\ - }\\
\midrule
0.077 & \makecell{SCG \\ SSR} & \makecell{(3, 846) \\ (1, 699)} & \makecell{(1.67, 1.17) \\ -} &  \makecell{(4,~846) \\ (2, 792)}  & \makecell{(1.5,~1.06) \\ - } \\
\midrule
0.075 & \makecell{SCG \\ SSR} & \makecell{(2, 846) \\ (1, 0)} & \makecell{(1.5, 2) \\ -} &  \makecell{(4,~846) \\ (2, 771)}  & \makecell{(1.5,~1.09) \\ - } \\
\midrule
0.074 & \makecell{SCG \\ SSR} & \makecell{(2, 846) \\ (1, 0)} & \makecell{(1.5, 2) \\ -} &  \makecell{(4,~846) \\ (2, 754)}  & \makecell{(1.5,~1.1) \\ - } \\
\midrule
0.073 & \makecell{SCG \\ SSR} & \makecell{(2, 846) \\ (1, 0)} & \makecell{(1.5, 2) \\ -} &  \makecell{(3,~846) \\ (2, 753)}  & \makecell{(1.33,~1.1) \\ - } \\
\midrule
0.07 & \makecell{SCG \\ SSR} & \makecell{(1, 846) \\ (1, 0)} & \makecell{(1, 2) \\ -} &  \makecell{(3,~846) \\ (2, 692)}  & \makecell{(1.33,~1.18) \\ - } \\
\midrule
0.067 & \makecell{SCG \\ SSR} & \makecell{(1, 846) \\ (1, 0)} & \makecell{(1, 2) \\ -} &  \makecell{(3~846) \\ (2, 0)}  & \makecell{(1.33,~2) \\ - } \\
\midrule
0.063 & \makecell{SCG \\ SSR} & \makecell{(1, 846) \\ (1, 0)} & \makecell{(1, 2) \\ -} &  \makecell{(3,~846) \\ (2, 0)}  & \makecell{(1.33,~2) \\ - } \\
\midrule
0.06 & \makecell{SCG \\ SSR} & \makecell{(1, 846) \\ (1, 0)} & \makecell{(1, 2) \\ -} &  \makecell{(4,~846) \\ (1, 0)}  & \makecell{(1.75,~2) \\ - } \\
\midrule
0.055 & \makecell{SCG \\ SSR} & \makecell{(1, 846) \\ (1, 0)} & \makecell{(1, 2) \\ -} &  \makecell{(2,~846) \\ (1, 0)}  & \makecell{(1.5,~2) \\ - }\\
\midrule
0.05 & \makecell{SCG \\ SSR} & \makecell{(4, 846) \\ (0, 0)} & \makecell{(2, 2) \\ -} &  \makecell{(1,~846) \\ (1, 0)}  & \makecell{(1,~2) \\ - }\\
\midrule
0.04 & \makecell{SCG \\ SSR} & \makecell{(0, 846) \\ (0, 0)} & \makecell{(\texttt{NA}, 2) \\ -} &  \makecell{(0,~846) \\ (0, 0)}  & \makecell{(\texttt{NA},~2) \\ - }\\

\bottomrule
\end{tabular}
\end{sc}
\end{small}
\end{center}
\vskip -0.1in
\end{table}

\end{document}